\documentclass[11pt]{article}
\usepackage{amssymb}
\usepackage{amsmath}
\usepackage{amsfonts}
\usepackage{textcomp}
\usepackage{amsthm}
\usepackage{bookmark}
\usepackage{arydshln}
\usepackage{dsfont}
\usepackage{changes}%\usepackage[final]{changes}  %\listofchanges
\usepackage{color}
\usepackage{verbatim}
\usepackage{cite}
\usepackage{enumerate}
\usepackage{cases}
\usepackage{mathrsfs}
\date{}

\newtheorem{Theorem}{Theorem}[section]
\newtheorem{Lemma}{Lemma}[section]
\newtheorem{Remark}{Remark}[section]

\newtheorem{Proposition}{Proposition}[section]

\numberwithin{equation}{section} \theoremstyle{plain}

\def\R{\mathbb{R}^3}
\def\H{\mathcal{H}}

\def\p{\partial}

\def\f{\frac}
\def\N{\mathbb N}

\def\la{\lambda}

\def\loc{\text{loc}}

\title{Qualitative  Properties of Ground States for the Stationary Magnetopolaron  with a Weak Magnetic Field}
\author{
 Yujin Guo\thanks{School of Mathematics and Statistics, and Key Laboratory of Nonlinear Analysis $\&$ Applications (Ministry of Education), Central China Normal University, Wuhan 430079, P. R. China.  Y. J. Guo is partially supported by NSF of China (Grants 12225106 and 12371113), and National Key R $\&$ D Program of China (Grant 2023YFA1010001). Email: \texttt{yguo@ccnu.edu.cn}. },
\, Shuang Wu\thanks{School of Mathematics and Statistics, Central China Normal University,  Wuhan 430079,
	P. R. China.  Email: \texttt{swu@mails.ccnu.edu.cn}.},
\, and\, Chenyi Yao\thanks{School of Mathematics and Statistics, Central China Normal University,  Wuhan 430079,
	P. R. China. Email: \texttt{yaochenyi@mails.ccnu.edu.cn}.}}
\date{\today}
\begin{document}
\maketitle
\begin{abstract}
We investigate ground states of  the stationary  magnetopolaron in $\mathbb R^3$ with a constant magnetic field. When the strength $|b|$ of the magnetic field is sufficiently small, we prove the uniqueness and nondegeneracy of ground states, up to magnetic translations and phase shifts. Applying the uniqueness result, we further derive rigorously the symmetry and monotonicity of ground states  for
sufficiently small $|b|>0$. The second-order asymptotic expansion of the ground state energy is also derived as $b\to 0$.
\end{abstract}

\noindent {\it Keywords:} The magnetopolaron; Weak magnetic fields; Ground states; Uniqueness; Energy expansions

%\bigskip

\section{Introduction}
In this paper, we study ground states of the following stationary magnetopolaron
\begin{equation}\label{equ}
-(\nabla+iA)^2u+u =\big(|x|^{-1}*|u|^2\big)u \ \,\ \text{in \ } \mathbb R^3,
\end{equation}
where $u:\mathbb R^3\to\mathbb C$, and the vector potential $A$ is chosen in the
symmetric gauge
\begin{equation}\label{1.1*}
A(x)=\frac b2(-x_2,x_1,0),  \ \ x=(x_1,x_2,x_3)\in\mathbb R^3 .
\end{equation}
For the choice \eqref{1.1*}, the corresponding magnetic field is 
\begin{equation*}
B=\nabla\times A=(0,0,b),
\end{equation*}
where $|b|\in [0, \infty)$ denotes the strength of the magnetic field.
Throughout the whole paper, we focus on ground states of \eqref{equ} in the regime
where the strength $|b|>0$ is sufficiently small.

The magnetopolaron \eqref{equ} arises from the magnetic Pekar model, which provides an effective description of a Fr\"ohlich polaron in an external magnetic field. The magnetic term of \eqref{equ} describes the coupling of the magnetopolaron to the external magnetic field, while the nonlocal convolution term $\big(|x|^{-1}*|u|^2\big)u$ of \eqref{equ} accounts for the Coulomb-type self-interaction induced by the polarization field, see \cite{Est,Dev,GL,GHW,FG} and the references therein for more physical motivations of \eqref{equ}. From a mathematical point of view, the magnetic covariant derivative of \eqref{equ} breaks the usual translation invariance, and the nonlocal interaction of \eqref{equ} brings extra difficulties to the analysis. Especially, when $A=0$,  equation \eqref{equ} reduces to the classical Choquard-Pekar equation
\begin{equation}\label{w}
-\Delta u+u=\big(|x|^{-1}*|u|^2\big)u \ \ \hbox{in \ } \mathbb R^3 ,\ \ u\in  H^1(\mathbb R^3, \mathbb C).
\end{equation}
The uniqueness and other analytical properties of ground states for (\ref{w}) were analyzed in \cite{Lieb76,Lieb01,Len,MV} and the references therein, see also Remark \ref{remark:2.1} of the present paper. In particular, the nondegeneracy of ground states for (\ref{w}) was proved in \cite{Len}. This result verifies a key
spectral assumption arising in the analysis of effective solitary-wave
dynamics for nonrelativistic Hartree equations (cf. \cite{FGJS}).

We now define the ground state energy of \eqref{equ} by
\begin{equation}\label{inf}
e(b):=\inf\left\{I_b(u):\  u\in H_A^1(\mathbb R^3,\mathbb C)\setminus\{0\},\ \langle I_b'(u),u\rangle=0\right\},
\end{equation}
where the energy functional $I_b(u)$ satisfies
\begin{equation}\label{functional}
I_b(u)=\frac12\int_{\mathbb R^3}\big(|\nabla_Au|^2+|u|^2\big)dx-\frac14\iint_{\mathbb R^3\times\mathbb R^3}\frac{|u(x)|^2|u(y)|^2}{|x-y|}dxdy , \ \ \nabla_A u := \nabla u+iAu ,
\end{equation}
and the magnetic Sobolev space $H_A^1(\R,\mathbb{C})$ is defined by
\begin{equation*}
H_A^1(\mathbb R^3,\mathbb C):=\Big\{u\in L^2(\mathbb R^3,\mathbb C):\ \int_{\mathbb R^3}|\nabla_Au|^2dx<+\infty\Big\}
\end{equation*}
endowed with the inner product
\begin{equation*}
(u|v)_{H_A^1}= \int_{\R}\big[(\nabla_{A}u|\nabla_{A}v)+(u|v)\big]dx
:= \int_{\R}\big[\operatorname{Re}(\nabla_{A}u\cdot\overline{\nabla_{A}v})+\operatorname{Re}(u\bar{v})\big]dx.
\end{equation*}
A function $u\in H_A^1(\mathbb R^3,\mathbb C)$ is called a ground state of
\eqref{equ}, if $u$ is a weak solution of \eqref{equ} and satisfies $I_b(u)=e(b)$.
It deserves to remark that the presence of a magnetic field substantially complicates the analysis of ground states for \eqref{equ}. Although the existence, multiplicity and semiclassical behavior of solutions for magnetic Choquard-type equations have been investigated in a series of works (cf. \cite{AFY,CS,CC,JR1}), relatively few results are available on their uniqueness. More recently, the local uniqueness and cylindrical symmetry of prescribed-mass constrained minimizers for a magnetic Choquard equation were established in \cite{LLL} in the large-mass regime under suitable assumptions on the electric potential. On the other hand, the uniqueness and nondegeneracy of ground states for local nonlinear magnetic equations have been established in \cite{DB,An}. Nevertheless, to the best of our knowledge, the uniqueness and nondegeneracy of ground states for \eqref{equ} in the weak magnetic-field regime have not been established so far.  A main additional difficulty of addressing these issues lies in the nonlocal nature of the Coulomb nonlinearity.

Motivated by the above facts, the purpose of the present paper is to investigate qualitative properties of ground states for the magnetopolaron \eqref{equ} in the regime where the magnetic field strength $|b|>0$ is sufficiently small. More precisely, our primary goals are to establish the uniqueness (up to magnetic translations and phase shifts), nondegeneracy, symmetry and monotonicity properties of ground states for the magnetopolaron \eqref{equ}. We also derive the second-order asymptotic expansion of the ground state energy for the magnetopolaron \eqref{equ} as the magnetic field tends to zero.

\subsection{Main results}

The main purpose of this subsection is to introduce the main results of the present paper. We first notice the typical feature of \eqref{equ} that the usual translation invariance of \eqref{equ} is
replaced by the magnetic translation invariance. Moreover, we define for any solution $u$ of \eqref{equ} and any $a\in\mathbb R^3$,  
\begin{equation}\label{1.45}
\tau_{a,A}u(x)=e^{-iA(a)\cdot x}u(x-a),
\end{equation}
where the vector $A(\cdot)$ is as in (\ref{1.1*}). Note that the magnetic translation operator commutes with the covariant derivative $\nabla_A$, $i.e.$,
$\nabla_{A}\circ\tau_{a,A}=\tau_{a,A}\circ \nabla_A$. However, magnetic translations do not commute with each other, since
\begin{equation}\label{t}
\tau_{c,A}\circ\tau_{a,A}=e^{iA(a)\cdot c}\tau_{a+c,A},\ \ \forall a,c\in\R.
\end{equation}
Consequently, the uniqueness of ground states for \eqref{equ} must be formulated
modulo magnetic translations and constant phase shifts.

The first main result of the present paper concerns the following uniqueness and nondegeneracy of ground states for \eqref{equ} when $|b|>0$ is sufficiently small.

\begin{Theorem}\label{T1.1}
Suppose
$u$ and $v$ are ground states of \eqref{equ}, where $A=\frac b2(-x_2,x_1,0)$. Then there exists a sufficiently small constant $\varepsilon>0$ such that for any $|b|<\varepsilon$, there exist
$a\in\mathbb R^3$ and $\theta\in\mathbb R$ satisfying
\begin{equation}\label{1A:T1.1}
u\equiv e^{i\theta}\tau_{a,A}v \ \hbox{ \ in \ } \mathbb R^3.
\end{equation}
Moreover, if $w\in H_A^1(\mathbb R^3,\mathbb C)$ is a solution of
\begin{equation}\label{1:T1.1}
-\Delta_A w+w=\big(|x|^{-1}*|u|^2\big)w+2u\big(|x|^{-1}*(u|w)\big)\ \hbox{ \ in \ } \mathbb R^3,
\end{equation}
where $(u|w)=\operatorname{Re}(u\overline w)$ and $u$ is a ground state of \eqref{equ}, then $w$ satisfies
\begin{equation}\label{1B:T1.1}
w=c_1 \big(-\partial_{x_1}u-\frac{ib}{2}x_2u\big)+c_2 \big(-\partial_{x_2}u+\frac{ib}{2}x_1u\big)+c_3 (-\partial_{x_3}u)+\lambda iu
\end{equation}
for some constants $c_1,c_2,c_3,\lambda\in\mathbb R$.
\end{Theorem}

The proof of Theorem~\ref{T1.1} is mainly based on a compactness argument and a spectral perturbation analysis. More precisely, we shall first prove that up to suitable magnetic translations and phase shifts, any ground state of \eqref{equ} strongly converges as $b\to0$ to the unique ground state $u_0>0$ of the limiting equation \eqref{w}. The magnetic translation and phase parameters are then fixed by imposing appropriate orthogonality conditions.

To prove the uniqueness of Theorem~\ref{T1.1}, we then need to analyze the nonlocal operator $L_n$ defined in \eqref{Ln}, where $b_n\to 0$ as $n\to\infty$. Our argument is motivated by the spectral perturbation approach developed in \cite{DB} for the local power-law magnetic Schr\"odinger equation, but the nonlocal nature of the Coulomb nonlinearity leads to an essential new difficulty. Indeed, in the local setting of \cite{DB}, the corresponding operator obtained from the averaged derivative of the nonlinearity is a pointwise real-linear operator, whereas the derivative of the Coulomb nonlinearity  in the present problem is given by
\[
w\mapsto
\big(|x|^{-1}*|u|^2\big)w
+2u\big(|x|^{-1}*(u|w)\big),
\]
which contains a genuine nonlocal convolution term. 
Since $L_n$ contains such nonlocal convolution terms, in Section \ref{S3} we shall establish the compactness and the spectral convergence of the operator $L_n$ as $n\to\infty$. Using the nondegeneracy of $u_0$, we are thus able to identify the eigenspace of the limiting operator $L$ corresponding to the eigenvalue $1$, which is generated by the three translation modes and the phase mode. This finally yields the uniqueness of ground states for \eqref{equ} as $b\to 0$. On the other hand, to prove the nondegeneracy of Theorem~\ref{T1.1}, we shall employ the same spectral perturbation analysis as above for the compact operator, associated with the linearized equation \eqref{1:T1.1} at a ground state, and show that its eigenspace corresponding to the eigenvalue $1$ has no additional directions beyond the four natural symmetry modes.

As an application of Theorem \ref{T1.1}, we next characterize the symmetry and monotonicity properties of ground states for \eqref{equ} as $b\to0$. For sufficiently small $|b|>0$, we shall prove that up to a suitable magnetic translation, a ground state of \eqref{equ} can be reduced to a ground state of the following Choquard--Pekar equation without the magnetic cross term
\begin{equation}\label{A11}
 -\Delta u+(1+|A|^2)u =\big(|x|^{-1}*|u|^2\big)u\ \  \hbox{in \ } \mathbb R^3 ,\ \ u\in   \H_b,
\end{equation}
where the vector $A$ is as in (\ref{1.1*}), and the space $\H_b$ is defined by \eqref{Hb}. The uniqueness and symmetry of real-valued ground states for (\ref{A11}) are addressed in Lemma 4.2, from which we shall prove the following symmetry and monotonicity properties.

\begin{Theorem}\label{T1.2}
Suppose $u$ is a ground state of \eqref{equ}, where $A=\frac b2(-x_2,x_1,0)$. Then there exists a sufficiently small constant $\varepsilon>0$ such that for any $0<|b|<\varepsilon$,
\begin{equation*}
u=e^{i\theta}\tau_{x_0,A}U_b
\end{equation*}
holds for some $x_0\in\mathbb R^3$ and $\theta\in\mathbb R$, where $U_b>0$ is a unique (up to translations in the $x_3$-direction) ground state of \eqref{A11}.
Moreover, $U_b$ may be chosen to be radially symmetric and decreasing in
$(x_1,x_2)$ and symmetric and decreasing in $x_3$.
\end{Theorem}

The proof of Theorem~\ref{T1.2} mainly follows from variational rearrangement arguments and a reduction of the magnetic problem to the auxiliary problem \eqref{A11}. More precisely, the proof of Theorem~\ref{T1.2} proceeds in the following three steps.

In the first step of proving Theorem~\ref{T1.2}, we equivalently reformulate in Lemma~\ref{L4.1} the constraint ground state problem of \eqref{A11} into the following minimization problem
\begin{equation*}
\widetilde q(b):=\inf_{u\in\mathcal H_b\setminus\{0\}}\widetilde Q_b(u),
\ \ \widetilde Q_b(u):=\frac{\int_{\mathbb R^3}\bigl[|\nabla u|^2+(1+|A|^2)|u|^2\bigr]\,dx}{\Big(\iint_{\mathbb R^3\times\mathbb R^3}
\frac{|u(x)|^2|u(y)|^2}{|x-y|}\,dx\,dy\Big)^{1/2}},
\end{equation*}
where the space $\mathcal H_b$ is defined by \eqref{Hb}.
The crucial advantage of equivalently studying the problem $\widetilde q(b)$ lies in the fact that the quotient $\widetilde Q_b(u)$ does not increase under the rearrangements of $u$, so that the rearrangement arguments can be applied to the problem $\widetilde q(b)$. Note that $\widetilde q(b)$ however does not require the rearranged functions to be on the Nehari manifold. This overcomes successfully the essential difficulty that the Nehari constraint associated with \eqref{A11} is not preserved under symmetric rearrangements.

In the second step of proving Theorem~\ref{T1.2}, we first prove in Proposition~\ref{P} that  for sufficiently small $|b|>0$, ground states of \eqref{equ} and \eqref{A11} are equivalent to each other, up to suitable magnetic translations. Applying Theorem~\ref{T1.1}, this then yields the uniqueness of positive ground states for \eqref{A11}, up to translations in the $x_3$-direction. This uniqueness plays a crucial role in establishing the symmetry in the $x_3$-direction. In the local problem studied in \cite{An}, the evenness in $x_3$ is obtained by the unique continuation principle (cf. \cite{Ken}). Unfortunately, such an argument cannot be applied directly to \eqref{A11}, due to the nonlocal nature of the Coulomb nonlinearity. For this reason, we shall make full use of the above uniqueness result to investigate the symmetry in the $x_3$-direction of positive ground states for \eqref{A11} as $b\to 0$.

As the last step, we are able to finish the proof of Theorem~\ref{T1.2}. We apply the Schwarz rearrangement and the Steiner rearrangement in $(x_1,x_2)$ and $x_3$, respectively, to a positive ground state $U_b$ of \eqref{A11}, which corresponds to a minimizer of $\widetilde q(b)$ in view of the first step. Since $\widetilde Q_b$ does not increase under these rearrangements, the rearranged functions are again minimizers of $\widetilde q(b)$. The equality case of the transverse rearrangement yields the radial symmetry and monotonicity in $(x_1,x_2)$, while the uniqueness up to $x_3$-translations, together with the Steiner rearrangement, yields the symmetry and monotonicity in $x_3$. It thus follows from the second step that up to a suitable magnetic translation and a phase shift, any ground state of \eqref{equ} is represented by a positive ground state $U_b$, which is cylindrically symmetric and decreasing in $(x_1,x_2)$, symmetric and decreasing in $|x_3|$. This therefore finishes the proof of Theorem~\ref{T1.2}. We refer to Section \ref{S4} for the detailed proof of Theorem~\ref{T1.2}.

\begin{Remark}
It is shown from Theorem \ref{T1.2} that for any sufficiently small nonzero magnetic field, any ground state of \eqref{equ} is cylindrically symmetric in $(x_1,x_2)$ and even in $x_3$, up to a suitable magnetic translation. This is in contrast with  the full radial symmetry of the unique real-valued ground state $u_0$ for (\ref{w}) when $b=0$. This indicates that even a weak nonzero magnetic field changes the symmetry structure of ground states for \eqref{equ} from  the full radial symmetry into cylindrical symmetry around the $x_3$-axis.
\end{Remark}

Motivated by the investigations on the ground state energy of a strong magnetic polaron  in \cite{FG} and the references therein, we are finally concerned with the following asymptotic behavior of the ground state energy $e(b)$ in
the weak magnetic field regime $b\to0$.

\begin{Theorem}\label{T1.4}
The ground state energy  $e(b)$ defined by \eqref{inf} satisfies
\begin{equation}\label{ee}
e(b)=e(0)+\frac{b^2}{12}\int_{\mathbb R^3}|x|^2|u_0(x)|^2dx+o(b^2) \text{ \ as \ }b\to 0,
\end{equation}
where  $u_0=u_0(|x|)>0$  is the  unique  ground state of \eqref{w}.
\end{Theorem}

%Inspired by \cite{DB}, the proof of Theorem~\ref{T1.4} relies on a refined expansion of the even real-valued ground state $u_A$ of \eqref{equ} as $b\to0$, where Theorem~\ref{T1.2} is also employed. More precisely, the key step of proving  Theorem~\ref{T1.4} is to prove Proposition~\ref{7.1}, which is concerned with the following expansion
%\begin{equation}\label{1:thm3-M}
%u_A=u_0+w_A+o(b^2)\text{ \ in \ }H^1(\mathbb R^3) \text{ \
% as } \ b\to0,
%\end{equation}
%where $w_A$ is a suitable function, and $u_0$ denotes the positive real-valued radial ground state of \eqref{w}. Since the linearized operator $L_b$?? around $u_0$ has a nontrivial kernel generated by translations and phase invariance, comparing with the existing works, e.g.\cite{DB}, there appears a new difficulty in the proof of (\ref{1:thm3-M})???. To overcome this difficulty, we shall construct in Lemma~\ref{L7.2} the correction term $w_A$ of (\ref{1:thm3-M}) in the orthogonal complement of the kernel for the operator $L_b$??. The ground state energy \eqref{ee} is finally derived by applying the expansion (\ref{1:thm3-M}) and the relation $I_0'(u_0)=0$.

Inspired by \cite{DB}, the proof of Theorem~\ref{T1.4} relies on a refined
expansion of the even real-valued ground state $u_A$ of \eqref{equ} as
$b\to0$, where Theorem~\ref{T1.2} is also employed. More precisely, the key
step of proving Theorem~\ref{T1.4} is to establish  Proposition~\ref{7.1}, which gives
the following expansion
\begin{equation}\label{thm1.3:K}
u_A=u_0+w_A+o(b^2)\text{ \ in \ }H^1(\mathbb R^3)
\quad \text{as }\ b\to0,
\end{equation}
where $w_A$ is a suitable correction term, and  $u_0=u_0(|x|)>0$ denotes the unique   ground state of \eqref{w}. The main difficulty of proving (\ref{thm1.3:K}) is to obtain the desired estimate \begin{equation}\label{thm1.3:KK}
v_A:=u_A-u_0-w_A=o(b^2)\ \    \text{in }\, H^1(\R)\,\ \text{as}\,\  b\to 0,
\end{equation}
which is derived by analyzing the equation of $v_A$, together with
its corresponding linearized
estimate as $b\to 0$.  Compared with \cite{DB}, our argument directly yields (\ref{thm1.3:KK}) without requiring an additional pointwise decaying estimate for $v_A$. The ground state energy expansion \eqref{ee} is finally proved in Section \ref{S5} by applying the   expansion (\ref{thm1.3:K}) and the relation $I_0'(u_0)=0$.

%Although we focus on the isotropic whole-space model, many of the arguments developed here are expected to extend to other related Choquard-type models. For several related Choquard-type models, including the anisotropic Choquard--Pekar equation in $\mathbb R^3$ \cite{JR} and the Choquard--Pekar equation on a ball \cite{DR}, uniqueness, symmetry and nondegeneracy of ground states have been established in the nonmagnetic case. The present work suggests a possible approach to extending these qualitative results to the corresponding weak magnetic-field regimes.

This paper is organized as follows. In Section~\ref{S2}, we mainly prove the continuity of the ground state energy $e(b)$ and
the compactness of ground states for  \eqref{equ} as $b\to0$. Section~\ref{S3} is  then concerned with the proof
of Theorem \ref{T1.1}. In Section \ref{S4}, we shall prove   Theorem \ref{T1.2} on the symmetry and monotonicity
properties of ground states for \eqref{equ}  as $b\to0$.  Section~\ref{S5} is finally devoted to the proof of Theorem~\ref{T1.4} on the asymptotic behavior of the ground state energy $e(b)$ as $b\to0$.

\section{Existence of Ground States }\label{S2}
In this section, we first prove in Proposition \ref{ex} the existence of ground states for  \eqref{equ}, and we then address in Subsection \ref{subS2.1} the continuity of the ground state energy $e(b)$ and the compactness of ground states for  \eqref{equ} as $b\to0$.

We begin with some preliminary results. If $A(x)\in L_{loc}^2(\R, \mathbb{R}^3)$, then we have  (cf. \cite{Lieb01}) the following diamagnetic inequality
\begin{equation}\label{dmi}
|\nabla|u||\leq|\nabla_A  u|\ \text{ \ a.e. in \ }\R,\ \ u\in H_A^1(\R,\mathbb{C}),
\end{equation}
where $\nabla_A  u=(\nabla+iA)u$ is the covariant derivative. Moreover, the identity of (\ref{dmi}) holds, if and only if $\nabla_A  u=\operatorname{sign}(u)\nabla|u|$. To handle the nonlocal term of  \eqref{equ}, we also need the following two different types of Hardy-Littlewood-Sobolev inequality.

\begin{Lemma}
(cf. \cite{Lieb83}) Let $p,q>1$ and $0<s<3$ satisfy $\frac{1}{p}+\frac{s}{3}+\frac{1}{q}=2$, and suppose $f\in L^p(\R)$ and $h\in L^q(\R)$. Then there exists a sharp constant $C(p,q,s)>0$, independent of $f$ and $h$, such that
\begin{equation}\label{HLSI}
\Big|\iint_{\R\times\R}f(x)|x-y|^{-s}h(y)dxdy\Big|\leq C(p,q,s)\|f\|_p\|h\|_q.
\end{equation}
In particular, we have
\begin{equation}\label{hls}
    \Big\|\int_{\R}\f{f(y)}{|x-y|}dy\Big\|_q\leq C\|f\|_p,
\end{equation}
where $p,q>1$ satisfy $\f1q=\f1p-\f23$.
\end{Lemma}

We next recall some results of ground states for the following elliptic problem without magnetics
\begin{equation}\label{w1}
-\Delta u+u=u(|x|^{-1}*|u|^2)  \text{ \ in \ }\R, \ \ u\in H^1(\R,\mathbb{C}),
\end{equation}
which is a natural limiting equation of \eqref{equ} as $b\to 0$. It is well known that \eqref{w1} admits ground states $u \in H^1(\R,\mathbb{C})$ in the sense that $I_0(u)=e(0)$ and $I'_0(u)=0$, where $e(0)$ and $I_0(\cdot)$ are as in (\ref{inf}) and \eqref{functional} with $b=0$, respectively.

%, then there exist $\theta \in \mathbb{R}$ and $a\in \R$ such that $v=e^{i\theta}\tau_{a,0}u$ in $\R$.

\begin{Remark}\label{remark:2.1}
It yields from \cite{Lieb76} that up to translations, positive real-valued ground states of  \eqref{w1} must be unique and radially symmetric. One can further derive from
\cite[Theorem 7.8]{Lieb01} that the set of all
complex-valued ground states for \eqref{w1} satisfies
\begin{equation*}
\Big\{e^{i\theta}u_0(\cdot-y):\, \theta\in[0,2\pi),\ y\in\mathbb R^3\Big\},
\end{equation*}
where $u_0=u_0(|x|)>0$ is the unique positive ground state of  \eqref{w1}.
\end{Remark}

We also note from \cite{Len} and Remark \ref{remark:2.1} that  the unique (up to magnetic translations and phase shifts) ground state of \eqref{w1} is
 nondegenerate in the following sense.

\begin{Lemma}\label{P2}  (\cite{Len})
Suppose   $u\in H^1(\R,\mathbb{C})$ is a ground state of \eqref{w1}, in the sense that $I_0(u)=e(0)$ and $I'_0(u)=0$. If the function $w\in H^1(\R,\mathbb{C})$ satisfies
\begin{equation}\label{12}
-\Delta w+w=(|x|^{-1}*|u|^2)w+2\big(|x|^{-1}*(u|w)\big)u\, \text{ \ in \ }\R,
\end{equation}
where $(u|w)=\operatorname{Re}(u\bar w)$, then there exist $y\in \R$ and $\la \in \mathbb{R}$ such that
 \begin{equation}\label{13}
 	w=\nabla u\cdot y+i\la u\text{ \ in \ }\R.
 \end{equation}
 \end{Lemma}

Furthermore, if $u$ is a solution of  \eqref{w1} and the function $w$ given by \eqref{13} is a solution of the linearized problem \eqref{12}, then $u$ and $w$ are orthogonal in $H^1(\R,\mathbb{C})$, $i.e.$,
\begin{equation}\label{14}
	\int_{\R}\big[(\nabla u|\nabla w)+(u|w)\big]dx=0.
\end{equation}
In fact, it follows from \eqref{w1}  and \eqref{12}  that
\begin{equation*}
	\int_{\R}\big[(\nabla u|\nabla w)+(u|w)\big]dx-\int_{\R}(|x|^{-1}*|u|^2)(u|w)dx=0,
\end{equation*}
and
\begin{equation*}
	\int_{\R}\big[(\nabla u|\nabla w)+(u|w)\big]dx-3\int_{\R}(|x|^{-1}*|u|^2)(u|w)dx=0,
\end{equation*}
which then yield the orthogonality identity \eqref{14}.

The following proposition concerns the existence of ground states for $e(b)$ with $b\in\mathbb{R}$, together with  an equivalent characterization of $e(b)$.

\begin{Proposition}\label{ex}
If $A=\frac{b}{2}(-x_2,x_1,0)$, then  \eqref{equ} admits  a ground state  $u\in H_{A}^1(\R,\mathbb{C})$ in the sense that
\begin{equation*}
	I_{b}(u)=e(b) \ \text{and} \ I'_{b}(u)=0,
\end{equation*}
where $e(b)$ and $I_b(u)$ are as in \eqref{inf} and \eqref{functional}, respectively.
Moreover,
\begin{equation}\label{ec}
	e(b)=\f14\inf_{v\in  H_{A}^1(\R,\mathbb{C})\setminus \{0\}}Q_A(v)^2,
\end{equation}
where the functional $Q_A:H_{A}^1(\R,\mathbb{C})\setminus \{0\}\to \mathbb{R}$ is defined by
\begin{equation}\label{2.8A}
	Q_A(v)=\f{\int_{\R}(|\nabla_A  v|^2+|v|^2)dx}{\left(\iint_{\R\times \R}\f{|v(x)|^2|v(y)|^2}{|x-y|}dxdy\right)^{\f12}}.
\end{equation}
\end{Proposition}

\noindent{\bf Proof.}
For any $v\in H_A^1(\mathbb R^3,\mathbb C)\setminus\{0\}$, denote
\begin{equation}\label{ABv}
A_v=\int_{\mathbb R^3}\big(|\nabla_Av|^2+|v|^2\big)dx,
\ \
B_v=\iint_{\mathbb R^3\times\mathbb R^3}\frac{|v(x)|^2|v(y)|^2}{|x-y|}dxdy.
\end{equation}
Then there exists a unique constant
\begin{equation*}
t_v=\left(\frac{A_v}{B_v}\right)^{1/2}>0
\end{equation*}
such that $t_vv$ satisfies the Nehari constraint
\begin{equation*}
\langle I_b'(t_vv),t_vv\rangle=0.
\end{equation*}
Moreover,
\begin{equation*}
I_b(t_vv)
=
\frac14\frac{A_v^2}{B_v}
=
\frac14 Q_A(v)^2 .
\end{equation*}
where $Q_A(v)$ is defined by (\ref{2.8A}).
Therefore, we have
\begin{equation}\label{2.85}
e(b)
=
\inf_{\substack{u\in H_A^1(\mathbb R^3,\mathbb C)\setminus\{0\}\\
\langle I_b'(u),u\rangle=0}} I_b(u)
=
\frac14
\inf_{v\in H_A^1(\mathbb R^3,\mathbb C)\setminus\{0\}} Q_A(v)^2 .
\end{equation}

We next claim that there exists a function $u\in H_A^1(\R,\mathbb{C})$ such that
\begin{equation}\label{qav}
Q_A(u)=\inf_{v\in H_{A}^1(\R,\mathbb{C})\setminus \{0\}}Q_A(v).
\end{equation}
To prove the claim (\ref{qav}), we note that
\begin{equation*}
\begin{aligned}
&\inf_{v\in H_{A}^1(\R,\mathbb{C})\setminus \{0\}}Q_A(v)\\
=&\inf_{v\in H_{A}^1(\R,\mathbb{C})\setminus \{0\}}\left\{\int_{\R}(|\nabla_A v|^2+|v|^2)dx:\,\iint_{\R\times\R}\f{|v(x)|^2|v(y)|^2}{|x-y|}dxdy=1\right\}\geq0.
\end{aligned}
\end{equation*}
Motivated by \cite[Theorem 3.1]{Est}, we denote
\begin{equation*}
q_b(\lambda)=\inf_{v\in H_{A}^1(\R,\mathbb{C})\setminus \{0\}}\left\{\int_{\R}(|\nabla_A v|^2+|v|^2)dx:\iint_{\R\times\R}\f{|v(x)|^2|v(y)|^2}{|x-y|}dxdy=\lambda\right\},\ \ \lambda>0,
\end{equation*}
so that $q_b(\lambda)=\sqrt{\lambda}q_b(1)$. This thus implies that if $q_b(1)> 0$, then we have
\begin{equation}\label{binding}
q_b(1-\lambda)+q_b(\lambda)=\sqrt{1-\lambda}q_b(1)+\sqrt{\lambda}q_b(1)>q_b(1),\ \ \forall  \, \lambda\in(0,1).
\end{equation}
 Let $\{v_n\}\subset H_A^1(\R,\mathbb{C})$ be a minimizing sequence of $q_b(1)$, $i.e.$, $\{v_n\}$ satisfies
\begin{equation}\label{minimiz}
\iint_{\R\times\R}\f{|v_n(x)|^2|v_n(y)|^2}{|x-y|}dxdy=1\ \ \forall   \, n\in\N^+,\text{ \ and \ }\lim_{n\to\infty}Q_A(v_n)=q_b(1).
\end{equation}
It is clear that $\{v_n\}$ is bounded uniformly in $H_A^1(\R,\mathbb C)$. Define
\begin{equation*}
f_n(r):=\sup_{y\in\R}\int_{B_r(y)}|v_n|^2dx, \ \ r>0.
\end{equation*}
If $\liminf_{n\to\infty} f_n(r)=0$ holds for every $r>0$, then we derive from \cite[Lemma 1.21]{Willem} that up to  a subsequence if necessary,
\begin{equation*}
v_n\to0\text{ \ strongly in \ }L^p(\R)\text{ \ as \ }n\to\infty,\ \ 2<p<6.
\end{equation*}
This then implies from \eqref{HLSI} that
\begin{equation*}
\lim_{n\to\infty}\iint_{\R\times\R}\f{|v_n(x)|^2|v_n(y)|^2}{|x-y|}dxdy=0,
\end{equation*}
which however contradicts with \eqref{minimiz}.  Therefore, there exists some $r>0$ such that
\begin{equation*}
q_b(1)\geq\liminf_{n\to\infty}f_n(r)>0.
\end{equation*}
This implies that there exist $\{a_n\}\subset\R$ and $0\not\equiv v\in H_A^1(\R,\mathbb C)$ such that up to a subsequence if necessary,
\begin{equation}\label{2.15M}
\tau_{a_n,A}v_n\rightharpoonup v\text{ \ weakly in \ }H_A^1(\R,\mathbb C)\text{ \ as \ }n\to\infty.
\end{equation}

We now claim that
\begin{equation}\label{2.14M}
0<\iint_{\R\times\R}\f{|v(x)|^2|v(y)|^2}{|x-y|}dxdy=1.
\end{equation}
On the contrary, suppose
\begin{equation*}
0<\lambda:=\iint_{\R\times\R}\f{|v(x)|^2|v(y)|^2}{|x-y|}dxdy<1.
\end{equation*}
It then follows from \cite[Lemma 2.2]{Frank14} and (\ref{2.15M}) that
\begin{equation}\label{dic}
\begin{aligned}
1=&\iint_{\R\times\R}\f{|\tau_{a_n,A}v_n(x)|^2|\tau_{a_n,A}v_n(y)|^2}{|x-y|}dxdy\\
= &\iint_{\R\times\R}\f{|\tau_{a_n,A}v_n(x)-v(x)|^2|\tau_{a_n,A}v_n(y)-v(y)|^2}{|x-y|}dxdy\\
&+\iint_{\R\times\R}\f{|v(x)|^2|v(y)|^2}{|x-y|}dxdy
+o(1) \text{ \ as \ }n\to\infty.
\end{aligned}
\end{equation}
Similar to \cite{Est}, we then obtain from \eqref{dic} that
\begin{equation}
q_b(1)\geq q_b(\lambda)+q_b(1-\lambda),\ \ 0<\lambda <1.
\end{equation}
This however contradicts with \eqref{binding}. Therefore, the claim (\ref{2.14M}) holds true.

We now conclude from above that $v$ is a minimizer of $q_b(1)$ by the weakly lower semi-continuity, which hence proves the claim \eqref{qav}.
Further, let $v\in H_A^1(\R,\mathbb C)\setminus\{0\}$ be a minimizer of $Q_A$, and set
$u=t_vv$, where $t_v=(A_v/B_v)^{1/2}$, and $A_v, B_v>0$ are defined in \eqref{ABv}. We then get that $u$ belongs to the Nehari
manifold and satisfies
\begin{equation*}
I_b(u)=\frac14Q_A(v)^2=e(b).
\end{equation*}
Since the Nehari manifold is a natural constraint, we have $I_b'(u)=0$, and $u$ is a ground state of \eqref{equ}. This proves that the infimum (\ref{ec})  is attained, which therefore completes the proof of Proposition \ref{ex}.\qed

\subsection{Compactness of ground states  as $b\to0$} \label{subS2.1}
The purpose of this subsection is to analyze the continuity of the ground state energy
$e(b)$ defined by \eqref{inf} and the compactness of ground states for \eqref{equ} as $b\to0$. We note that the first variation of $I_b$ defined by (\ref{functional}) satisfies
\begin{equation}\label{3.02}
\left\langle I_b'(u),v\right\rangle=\int_{\R}\big[(\nabla_Au|\nabla_Av)+(u|v)-(|x|^{-1}*|u|^2)(u|v)\big]dx,\ \ \forall u,v\in H_A^1(\R,\mathbb C),
\end{equation}
where $(u|v)=\operatorname{Re}(u\bar v)$. The main results of this subsection can be then stated as the following proposition.

\begin{Proposition}\label{P3.1}
Let the ground state energy $e(b)$ be defined by \eqref{inf}. Then we have
\begin{enumerate}
    \item  The energy $e(b)$ is continuous at $b=0$.
    \item Suppose $0\not\equiv u_n\in H_{A_n}^1(\R,\mathbb{C})$ is a ground state of \eqref{equ}, where $A=A_n=\f{b_n}{2}(-x_2,x_1,0)$ and $b_n\to0$ as $n\to\infty$, in the sense that
\begin{equation}\label{2.155}
I'_{b_n}(u_n)=0\  \text{ and } \ I_{b_n}(u_n)=e(b_n).
\end{equation}
Then there exist a subsequence, still denoted by $\{u_n\}$, of $\{u_n\}$,  a sequence $\{c_n\} \subset\R$ and a function  $u\in H^1(\R,\mathbb{C})$ satisfying $I_0(u)=e(0)$ and $I'_0(u)=0$ such that
\begin{equation}\label{strongcon}
\tau_{c_n,A_n}u_n\to u\text{ \ and \ }\nabla_{A_n}(\tau_{c_n,A_n}u_n)\to \nabla u\text{ \ strongly in\ }L^2(\R)\text{\  as\ }n\to \infty.
\end{equation}
\end{enumerate}
\end{Proposition}

In order to prove Proposition \ref{P3.1}, we  first prove the following lemma.

\begin{Lemma}\label{L3.1}
Suppose the sequence $\{u_n\}\subset H_{A_n}^1(\R,\mathbb{C})\setminus\{0\}$ satisfies
\begin{equation}\label{3.01}
I'_{b_n}(u_n)=0 \ \ \forall \, n\in\N^+,\ \ \text{and } \ \limsup_{n\to \infty} I_{b_n}(u_n)<+\infty,
\end{equation}
where $A_n:=\frac{b_n}{2}(-x_2,x_1,0)$ and $b_n\to0$ as $n\to\infty$.
Then there exist a subsequence, still denoted by $\{u_n\}$, of $\{u_n\}$,  a sequence $\{c_n\} \subset\R$ and a function  $0\not\equiv u\in H^1(\R,\mathbb{C})$ such that
\begin{equation}\label{weaklim}
\tau_{c_n,A_n}u_n \rightharpoonup u ,\ \nabla_{A_n}(\tau_{c_n,A_n}u_n)\rightharpoonup \nabla u\text{ \ weakly in \ }L^2(\R)\text{ \ as  \ }n \to \infty ,
\end{equation}
and
\begin{equation}\label{solution}
I_0'(u)=0.
\end{equation}
Moreover, we have
\begin{equation}\label{lowercon}
	\liminf_{n\to \infty}I_{b_n}(u_n)\geq I_0(u).
\end{equation}
\end{Lemma}

\noindent{\bf Proof.}
We first derive from \eqref{3.02} and \eqref{3.01} that
 \begin{equation*}
 	\begin{aligned}
 	\f14\int_{\R}\big(|\nabla_{A_n}u_n|^2+|u_n|^2\big)dx&=I_{b_n}(u_n),
 	\end{aligned}
 \end{equation*}
which then yields  that the sequence $\{u_n\}$ is bounded uniformly in $H_{A_n}^1(\R,\mathbb{C})$ as $n\to\infty$, $i.e.$,
\begin{equation}\label{bound}
	\limsup_{n\to \infty}\int_{\R}\big(|\nabla_{A_n}u_n|^2+|u_n|^2\big)dx<+\infty.
\end{equation}
By Hardy-Littlewood-Sobolev inequality \eqref{HLSI}, we have
 \begin{equation*}
 		\iint_{\R\times\R}\f{|u_n(x)|^2|u_n(y)|^2}{|x-y|}dxdy\leq C\|u_n\|_{\f{12}{5}}^4.
  \end{equation*}
Applying \cite[Lemma 2.3]{Mor}, we then deduce from \eqref{dmi} that
\begin{equation}\label{H}
\begin{aligned}
	&\iint_{\R\times\R}\f{|u_n(x)|^2|u_n(y)|^2}{|x-y|}dxdy\\
    \leq& C\left(\sup_{y\in \R}\int_{B_1(y)}|u_n|^{\f{12}{5}}dx\right)^{\f{5}{18}}\left(\int_{\R}\big(\big|\nabla|u_n|\big|^2+|u_n|^2\big)dx\right)^{\f53}\\
	\leq& C\left(\sup_{y\in \R}\int_{B_1(y)}|u_n|^{\f{12}{5}}dx\right)^{\f{5}{18}}\left(\int_{\R}\big(|\nabla_{A_n}u_n|^2+|u_n|^2\big)dx\right)^{\f53},
\ \ \forall  n\in \mathbb{N}^+.
\end{aligned}
\end{equation}
On the other hand, we obtain from \eqref{3.02} and \eqref{3.01} that
\begin{equation}\label{HH}
	\begin{aligned}
	\int_{\R}\big(|\nabla_{A_n}u_n|^2+|u_n|^2\big)dx&=	 \iint_{\R\times\R}\f{|u_n(x)|^2|u_n(y)|^2}{|x-y|}dxdy+\left<I'_{b_n}(u_n),u_n\right>\\&
	=\iint_{\R\times\R}\f{|u_n(x)|^2|u_n(y)|^2}{|x-y|}dxdy,
\ \ \forall  n\in \mathbb{N}^+.
	\end{aligned}
\end{equation}
It then yields from \eqref{H} and \eqref{HH} that
\begin{equation}\label{H1}
\begin{aligned}
&\iint_{\R\times\R}\f{|u_n(x)|^2|u_n(y)|^2}{|x-y|}dxdy\\ \leq& C\left(\sup_{y\in \R}\int_{B_1(y)}|u_n|^{\f{12}{5}}dx\right)^{\f{5}{18}}\left(\iint_{\R\times\R}\f{|u_n(x)|^2|u_n(y)|^2}{|x-y|}dxdy\right)^{\f53} ,
\ \ \forall  n\in \mathbb{N}^+.
\end{aligned}
\end{equation}
Applying \eqref{bound}, this thus implies that
\begin{equation*}
	\liminf_{n\to \infty}\left(\sup_{y\in \R}\int_{B_1(y)}|u_n|^{\f{12}{5}}dx\right)^{\f{5}{18}}>0.
\end{equation*}
Hence, there exists a sequence $\{c_n\}\subset\R$ such that
\begin{equation}\label{H2}
	\liminf_{n\to \infty}\int_{B_1(c_n)}|u_n|^{\f{12}{5}}dx>0.
\end{equation}
We thus conclude from \cite{Lions} that the sequence $\{u_n\}$ cannot vanish. Therefore, there exists a magnetic translation $\tau_{c_n,A_n}$ such that
\begin{equation}\label{nonvan}
	\liminf_{n\to \infty}\int_{B_1(0)}|\tau_{c_n,A_n}u_n|^{\f{12}{5}}dx>0.
\end{equation}

It follows from (\ref{1.45}) and \eqref{bound} that
\begin{equation*}
\liminf_{n\to \infty}\int_{\R}\big(|\nabla_{A_n}(\tau_{c_n,A_n}u_n)|^2+|\tau_{c_n,A_n}u_n|^2\big)dx=\liminf_{n\to \infty}\int_{\R}\big(|\nabla_{A_n}u_n|^2+|u_n|^2\big)dx<+\infty.
\end{equation*}
Applying \cite[Lemmas 2.2 and 2.3]{DB}, it then yields from \eqref{nonvan} that there exist $0\not\equiv u\in H^1(\R,\mathbb{C})$ and a subsequence, still denoted by $\{u_n\}$, of $\{u_n\}$ such that
\begin{equation*}
\tau_{c_n,A_n}u_n\rightharpoonup u\ \text{ and }\ \nabla_{A_n}(\tau_{c_n,A_n}u_n)\rightharpoonup \nabla  u\ \text{ weakly in}\ L^2(\R)\text{ \ as \ }n\to\infty,
\end{equation*}
which proves \eqref{weaklim}. We then derive from \eqref{weaklim} that for every $\phi \in C_c^1(\R,\mathbb{C})$,
\begin{equation}\label{01}
\nabla_{A_n}\phi \to \nabla  \phi\text{ \ strongly in \ } L^2(\R)\text{ \ as \ }n\to\infty,
\end{equation}
and
\begin{equation}\label{02}
\begin{aligned}
&\lim_{n\to\infty}\iint_{\R\times \R}\f{|\tau_{c_n,A_n}u_n(y)|^2\big(\tau_{c_n,A_n}u_n(x)|\phi(x)\big)}{|x-y|}dxdy\\
=& \iint_{\R\times\R}\f{|u(y)|^2\big(u(x)|\phi(x)\big)}{|x-y|}dxdy.
\end{aligned}
\end{equation}
Since $\tau_{c_n,A_n}u_n$ is also a critical point of $I_{b_n}$, we obtain that
\begin{equation*}
\left\langle I'_{b_n}(\tau_{c_n,A_n}u_n),\phi\right\rangle=0.
\end{equation*}
 By the weak convergence of \eqref{weaklim}, we thus deduce from   \eqref{01} and \eqref{02} that
\begin{equation*}
\left\langle I'_0(u),\phi\right\rangle
=
\lim_{n\to\infty}
\left\langle I'_{b_n}(\tau_{c_n,A_n}u_n),\phi\right\rangle
=0,
\end{equation*}
which thus proves \eqref{solution}.

The rest is to prove \eqref{lowercon}.
 In fact, it follows from \eqref{functional} and \eqref{3.02} that
 \begin{equation*}
 	I_{b_n}(u_n)=\f14\iint_{\R\times\R}\f{|u_n(x)|^2|u_n(y)|^2}{|x-y|}dxdy+\f12\left<I'_{b_n}(u_n),u_n\right>,\ \ \forall\, n\in \mathbb{N}^+.
 \end{equation*}
Since $I'_{b_n}(u_n)= 0$ holds for all $n\in \mathbb{N}^+$, it follows from \eqref{bound} that
 \begin{equation}\label{A}
 	\begin{aligned}
 	\liminf_{n\to \infty}I_{b_n}(u_n)=&\f14 \liminf_{n\to \infty}\iint_{\R\times\R}\f{|u_n(x)|^2|u_n(y)|^2}{|x-y|}dxdy\\
    =&\f14\liminf_{n\to \infty}\iint_{\R\times\R}\f{|\tau_{c_n,A_n}u_n(x)|^2|\tau_{c_n,A_n}u_n(y)|^2}{|x-y|}dxdy\\
    \geq&\f14\iint_{\R\times\R}\f{|u(x)|^2|u(y)|^2}{|x-y|}dxdy,
     \end{aligned}
 \end{equation}
 where   Fatou's lemma is used for the inequality.
Since $I_0'(u)=0$, we obtain from \eqref{A} and above that
\begin{equation*}
I_{0}(u)=\f14\iint_{\R\times\R}\f{|u(x)|^2|u(y)|^2}{|x-y|}dxdy+\f12\left<I'_{0}(u),u\right>\leq\liminf_{n\to \infty}I_{b_n}(u_n),
\end{equation*}
which proves \eqref{lowercon}. Therefore, it completes the proof of Lemma \ref{L3.1}. \qed

\vspace{5pt}

Applying Lemma \ref{L3.1}, we now address the proof of Proposition \ref{P3.1}.

\vspace{5pt}

\noindent{\bf Proof of Proposition \ref{P3.1}.}
Since $C_c^1(\R,\mathbb{C})$ is dense (cf. \cite[Theorem 7.22]{Lieb01}) in $H_A^1(\R,\mathbb{C})$, we deduce from Proposition \ref{ex} that
\begin{equation}\label{2.33M}
	e(b)=\f14\inf_{v\in  C_c^1(\R,\mathbb{C})\setminus \{0\}}Q_A(v)^2.
\end{equation}
For any $\eta>0$, there exists $v_\eta\in C_c^1(\R,\mathbb C)\setminus\{0\}$ such that
\begin{equation*}
\frac14 Q_0(v_\eta)^2\le e(0)+\eta.
\end{equation*}
The variational characterization of (\ref{2.33M}) thus yields that
\begin{equation*}
e(b)\le \frac14 Q_{A}(v_\eta)^2.
\end{equation*}
Since the map $b\mapsto Q_{A}(v)$ is continuous for any fixed $v\in C_c^1(\R,\mathbb C)\setminus\{0\}$, we have
\begin{equation*}
Q_{A}(v_\eta)\to Q_0(v_\eta) \text{ \ as \ }b\to0,
\end{equation*}
which then gives that
\begin{equation*}
\limsup_{b\to0}e(b)
\le
\frac14Q_0(v_\eta)^2
\le e(0)+\eta.
\end{equation*}
Since $\eta>0$ is arbitrary, we further obtain from above that
\begin{equation*}
\limsup_{b\to0}e(b)\le e(0).
\end{equation*}
This implies that $e(b)$ is upper semicontinuous at $b=0$.

We now claim that $e(b)$ is lower semicontinuous at $b=0$. By Lemma \ref{L3.1}, there exist a subsequence, still denoted by $\{u_n\}$, of $\{u_n\}$, a sequence $\{c_n\}\subset \R$ and $0\not\equiv u\in H^1(\R,\mathbb{C})$ satisfying $I'_0(u)=0$ such that
\begin{equation}\label{2.34M}
	\tau_{c_n,A_n}u_n\rightharpoonup u\ \text{ and }\ \nabla_{A_n}(\tau_{c_n,A_n}u_n)\rightharpoonup \nabla  u\ \, \text{weakly in}\ L^2(\R) \,\ \text{as }\ n \to \infty,
\end{equation}
and
\begin{equation*}
    \liminf_{n\to \infty}I_{b_n}(u_n)\geq I_0(u).
\end{equation*}
Since $I_0(u)\geq e(0)$, we have
\begin{equation}\label{eE}
	\liminf_{n\to \infty}e(b_n)=\liminf_{n\to\infty}I_{b_n}(u_n)\geq I_0(u)\geq e(0).
\end{equation}
This then proves the claim that $e(b)$ is lower semicontinuous at $b=0$. Hence, $e(b)$  is continuous at $b=0$, and $I_0(u)=e(0)$. Proposition \ref{P3.1}(1) is thus proved.

We finally prove \eqref{strongcon}. Since $e(b)$ is continuous at $b=0$, we have $\lim_{n\to \infty}I_{b_n}(u_n)=\lim_{n\to \infty}e(b_n)=e(0)=I_0(u)$. Since
\begin{equation*}
	 \f14\int_{\R}\big(|\nabla_{A_n}(\tau_{c_n,A_n}u_n)|^2+|\tau_{c_n,A_n}u_n|^2\big)dx=I_{b_n}(u_n)-\f14\left<I'_{b_n}(u_n),u_n\right>,
\end{equation*}
we obtain from \eqref{2.155} that
\begin{equation*}
	\begin{aligned}
	 &\f14\lim_{n\to\infty}\int_{\R}\big(|\nabla_{A_n}(\tau_{c_n,A_n}u_n)|^2+|\tau_{c_n,A_n}u_n|^2\big)dx\\
    =&\lim_{n\to\infty}I_{b_n}(u_n)=I_0(u)=\f14\iint_{\R\times\R}\f{|u(x)|^2|u(y)|^2}{|x-y|}dxdy\\
    =&\f14\int_{\R}\big(|\nabla u|^2+|u|^2\big)dx,
	\end{aligned}
\end{equation*}
where we have used  $I_0'(u)=0$ in the last identity.
This thus proves \eqref{strongcon} in view of (\ref{2.34M}). It therefore completes the proof of Proposition \ref{P3.1}(2), and we are done. \qed

\section{Uniqueness and Nondegeneracy of Ground States}\label{S3}
In this section, we first derive some estimates of ground states for  \eqref{equ} as $b\to 0$. In Subsection 3.1, we then prove Theorem \ref{T1.1} on the uniqueness, up to magnetic translations $\tau_{a,A}$ and rotations in $\mathbb{C}$, and nondegeneracy  of ground states for  \eqref{equ} as $b\to 0$. We start with the following lemma.

\begin{Lemma}\label{u1} Suppose $u_n$ is a ground state of \eqref{equ} with $A=A_n=\f{b_n}{2}(-x_2,x_1,0)$, where $b_n\to0$ as $n\to\infty$, and let $u_0=u_0(|x|)>0$ be the unique positive ground state of   \eqref{w}. Then there exist two sequences $\{\widetilde\theta_n\} \subset \mathbb{R}$ and $\{\widetilde a_n\} \subset \R$ such that
\begin{equation}\label{u11}
\lim_{n\to\infty}\int_{\R}\big(|\nabla_{A_n}(e^{i\widetilde\theta_n}\tau_{\widetilde a_n,A_n}u_n)-\nabla u_0|^2+|e^{i\widetilde \theta_n}\tau_{\widetilde a_n,A_n}u_n-u_0|^2\big)dx=0.
\end{equation}
Moreover,  the sequences $\{\widetilde\theta_n\}$ and $\{\widetilde a_n\}$ can be chosen so that for sufficiently large $n\in \mathbb{N}^+$,
\begin{equation}\label{u3}
\begin{aligned}
&\int_{\R}\Big\{\big[\nabla_{A_n}(e^{i\widetilde\theta_n}\tau_{\widetilde a_n,A_n}u_n)|\nabla(\nabla u_0\cdot y)\big]\\
&\qquad\ +\big[e^{i\widetilde \theta_n}\tau_{\widetilde a_n,A_n}u_n|\nabla u_0\cdot y\big]\Big\}dx=0 \  \ \forall \,  y\in\R,\\
&\int_{\R}\Big\{\big[\nabla_{A_n}(e^{i\widetilde \theta_n}\tau_{\widetilde a_n,A_n}u_n)|\nabla iu_0\big]+(e^{i\widetilde \theta_n}\tau_{\widetilde a_n,A_n}u_n|iu_0)\Big\}dx=0.
\end{aligned}
\end{equation}
\end{Lemma}

\noindent{\bf Proof.} We first claim that there exist two sequences $\{\theta_n\} \subset \mathbb{R}$ and $\{a_n\} \subset \R$ such that
\begin{equation}\label{u12}
\lim_{n\to\infty}\int_{\R}\big(|\nabla_{A_n}(e^{i\theta_n}\tau_{a_n,A_n}u_n)-\nabla u_0|^2+|e^{i\theta_n}\tau_{a_n,A_n}u_n-u_0|^2\big)dx=0.
\end{equation}
Indeed, it follows from Proposition \ref{P3.1} that there exist a subsequence, still denoted by $\{u_n\}$, and a sequence $\{c_n\}\subset\R$ such that
\begin{equation*}
\tau_{c_n,A_n}u_n\to \tilde u\text{ \ and \ }\nabla_{A_n}(\tau_{c_n,A_n}u_n)\to \nabla\tilde u\text{ \ strongly in\ }L^2(\R)\ \text{\  as\ }\ n\to \infty,
\end{equation*}
where   $\tilde{u}\in H^1(\R,\mathbb{C})$ is a ground state of   \eqref{w}.
Recall from Remark \ref{remark:2.1} that there exist $c\in\R$ and $k\in\mathbb{R}$ such that $\tilde{u}=e^{ik}\tau_{c,0}u_0$. We then derive from \eqref{t} that
\begin{equation}\label{4.2}
	\begin{split}
	o(1)=&\int_{\R}|\nabla_{A_n}(\tau_{c_n,A_n}u_n)-\nabla \tilde{u}|^2dx\\
    =&\int_{\R}|\nabla_{A_n}(\tau_{c_n,A_n}u_n)-e^{ik}\tau_{c,0}\nabla u_0|^2dx\\=&\int_{\R}|e^{-ik}\tau_{-c,A_n}\nabla_{A_n}(\tau_{c_n,A_n}u_n)-\tau_{-c,A_n}\tau_{c,0}\nabla u_0|^2dx\\=&\int_{\R}|e^{-ik}e^{-iA_n(c_n)\cdot c}\nabla_{A_n}(\tau_{c_n-c,A_n}u_n)-\tau_{-c,A_n}\tau_{c,0}\nabla u_0|^2dx\text{ \ as \ }n\to\infty.
	\end{split}
\end{equation}
On the other hand, by Lebesgue's dominated convergence theorem, we have
\begin{equation}\label{4.3}
\int_{\R}|\nabla u_0-\tau_{-c,A_n}\tau_{c,0}\nabla u_0|^2dx\to 0\text{ \ as}\,\ n \to \infty.
\end{equation}
We then obtain from \eqref{4.2} and \eqref{4.3} that
\begin{equation}\label{4.4}
\int_{\R}|e^{-ik}e^{-iA_n(c_n) \cdot c}\,\nabla_{A_n}(\tau_{c_n-c,A_n}u_n)-\nabla u_0|^2dx\to 0 \text{ \ as\ }\, n \to \infty.
\end{equation}
Similar to \eqref{4.4}, we also have
\begin{equation}\label{4.5}
	\int_{\R}|e^{-i[k+A_n(c_n)\cdot c]}\tau_{c_n-c,A_n}u_n-u_0|^2dx\to 0 \text{ \ as \ }  n \to \infty.
\end{equation}
Setting $\theta_n=-k-A_n(c_n)\cdot c$ and $a_n=c_n-c$, it then follows from above that \eqref{u12} holds for this subsequence. Since $u_0>0$ is unique, one can further verify that the claim \eqref{u12} essentially holds true for the whole sequence $\{u_n\}$.

Define the map $\psi_n\in C(\mathbb{R}^4,\mathbb{R}^4)$ by
\begin{equation*}
\begin{aligned}
&\psi_n(y,t)=\Big(\psi_n^{(1)}(y,t),\psi_n^{(2)}(y,t),\psi_n^{(3)}(y,t),\psi_n^{(4)}(y,t)\Big),
\end{aligned}
\end{equation*}
where
\begin{equation*}
\begin{aligned}
\psi_n^{(j)}(y,t)=&\int_{\R}\Big[\big(\nabla_{A_n}(e^{i({\theta}_n+t)}\tau_{-y,A_n}\tau_{{a}_n,A_n}u_n)|\nabla (\p_{x_j}u_0)\big)\\
&+\big(e^{i({\theta}_n+t)}\tau_{-y,A_n}\tau_{{a}_n,A_n}u_n|(\p_{x_j}u_0)\big)\Big]dx, \,\ j=1,2,3,
\end{aligned}
\end{equation*}
and
\begin{equation*}
\begin{aligned}
\psi_n^{(4)}(y,t)=\int_{\R}\Big[\big(\nabla_{A_n}(e^{i({\theta}_n+t)}\tau_{-y,A_n}\tau_{{a}_n,A_n}u_n) | \nabla iu_0\big)+\big((e^{i({\theta}_n+t)}\tau_{-y,A_n}\tau_{{a}_n,A_n}u_n)|iu_0\big)\Big]dx.
\end{aligned}
\end{equation*}
Here $a_n\in\R$ and $\theta_n\in\mathbb{R}$ are as in \eqref{u12}.
Since $\nabla_{A_n}\circ(\tau_{y,A_n}\tau_{{a}_n,A_n})=(\tau_{y,A_n}\tau_{{a}_n,A_n})\circ \nabla_{A_n}$, we then deduce from \eqref{u12} that  $\psi_n \to \psi$ strongly in $L^\infty_{\text{loc}}(\mathbb{R}^4)$ as $n\to\infty$, where the function $\psi\in C(\mathbb{R}^4,\mathbb{R}^4)$ satisfies
\begin{equation}\label{3.7A}
\psi(y,t)=\Big(\psi^{(1)}(y,t),\psi^{(2)}(y,t),\psi^{(3)}(y,t),\psi^{(4)}(y,t)\Big),
\end{equation}
together with
\begin{equation*}
\begin{aligned}
\psi^{(j)}(y,t)=&\int_{\R}\Big[\big(\nabla (e^{it}\tau_{-y,0}u_0)|\nabla (\p_{x_j}u_0)\big)+\big(e^{it}\tau_{-y,0}u_0| (\p_{x_j}u_0)\big)\Big]dx,\,\ j=1,2,3,
\end{aligned}
\end{equation*}
and
\begin{equation*}
\begin{aligned}
\psi^{(4)}(y,t)=\int_{\R}\Big[\big(\nabla(e^{it}\tau_{-y,0}u_0) | \nabla (iu_0)\big)+\big((e^{it}\tau_{-y,0}u_0)|iu_0\big)\Big]dx.
\end{aligned}
\end{equation*}

We derive from \eqref{14} and (\ref{3.7A}) that
\begin{equation}\label{3.6}
\psi(0,0)=(0,0,0,0).
\end{equation}
Moreover, direct calculations yield that
\begin{equation*}
\frac{\partial \psi^{(j)}}{\partial y_k}(0,0)=\int_{\R}\Big[\big(\nabla(\partial_{x_k}u_0)|\nabla(\partial_{x_j}u_0)\big)+\big(\partial_{x_k}u_0|\partial_{x_j}u_0\big)
\Big]dx,\ \ j,k=1,2,3;
\end{equation*}
\begin{equation*}
\frac{\partial \psi^{(4)}}{\partial y_k}(0,0)=0,\ \ k=1,2,3;\ \
\frac{\partial \psi^{(4)}}{\partial t}(0,0)=\int_{\R}\Big(|\nabla u_0|^2+|u_0|^2\Big)dx>0;
\end{equation*}
and
\begin{equation*}
\frac{\partial\psi^{(j)}}{\partial t}(0,0)=0,\ \ j=1,2,3.
\end{equation*}
Since $u_0>0$ is radially symmetric, the matrix
\begin{equation*}
D\psi(0,0)
\end{equation*}
is diagonal with  positive diagonal entries. This then implies that for sufficiently small $\eta>0$, the Brouwer degree
$\deg(\psi,B_\eta,0)$ is well-defined and
\begin{equation*}
\deg(\psi,B_\eta,0)=1.
\end{equation*}
Since $\psi_n \to \psi$ strongly in $L^\infty_{\text{loc}}(\mathbb{R}^4,\mathbb{R}^4)$ as $n\to\infty$, we obtain that deg($\psi_n,B_{\eta},0$)=1 holds for sufficiently large $n\in\N^+$. Thus, there exists a sequence $\{(y_n,t_n)\}\subset \R \times\mathbb{R} $, where  $(y_n,t_n)\to (0,0)$ as $n\to \infty$, such that
\begin{equation}\label{3.10A}
    \psi_{n}(y_n,t_n)=(0,0,0,0) \ \text{ holds for sufficiently large }\ n\in\N.
\end{equation}
 Setting $\widetilde a_n=a_n-y_n$ and $\widetilde \theta_n={\theta}_n+t_n-A_n({a}_n)\cdot y_n$, we then conclude that \eqref{u3} holds true. Applying again Lebesgue's dominated convergence theorem, we finally obtain that \eqref{u11} also holds true. This therefore completes the proof of Lemma \ref{u1}. \qed

\vspace{5pt}
To address the  uniqueness of Theorem \ref{T1.1}, by contradiction, we now suppose that for some sequence $\{b_n\}$ satisfying $b_n\to0$ as $n\to\infty$, there exist two different ground states $u_n$ and $v_n$ of \eqref{equ} with $A=A_n=\f{b_n}{2}(-x_2,x_1,0)$, which are not related by magnetic translations and constant phase factors.
We then deduce from Lemma \ref{u1} that there exist the sequences $\{\widetilde\theta_n^u, \widetilde a_n^u\}  $ and $\{\widetilde\theta_n^v, \widetilde a_n^v\}$ such that
\begin{equation}\label{tuv}
\tilde{u}_n:=e^{i\widetilde\theta_n^u}\tau_{\widetilde a_n^u,A_n}u_n \ \text{ and }  \ \tilde{v}_n:=e^{i\widetilde\theta_n^v}\tau_{\widetilde a_n^v,A_n}v_n
\end{equation}
satisfy \eqref{u11} and \eqref{u3}.
Since $\tilde u_n$ and $\tilde v_n$  are also solutions of \eqref{equ} with $A=A_n=\frac{b_n}{2}(-x_2,x_1,0)$, we obtain that
\begin{equation}\label{3.11}
-\Delta_{A_n}(\tilde u_n-\tilde v_n)+(\tilde u_n-\tilde v_n)=f(\tilde u_n)-f(\tilde v_n)\text{ \ in \ }\R,
\end{equation}
where $f(u):=u(|x|^{-1}*|u|^2)$. Note that
\begin{equation*}
\begin{aligned}
 f(\tilde u_n)-f(\tilde v_n)
=&\int_{0}^1\frac{d}{dt}f\big((1-t)\tilde{v}_n+t\tilde u_n\big)dt\\
=&\int_{0}^1Df\big((1-t)\tilde{v}_n+t\tilde u_n\big)[\tilde u_n-\tilde v_n]dt,
\end{aligned}
\end{equation*}
and denote the nonlocal operator
\begin{equation}\label{3.12M}
W_n:=\int_{0}^1Df\big((1-t)\tilde{v}_n+t\tilde u_n\big)dt,
\end{equation}
so that \eqref{3.11} can be rewritten as
\begin{equation}\label{3.12}
-\Delta_{A_n}(\tilde u_n-\tilde v_n)+(\tilde u_n-\tilde v_n)=W_n[\tilde u_n-\tilde v_n]\ \text{ \ in \ }\R.
\end{equation}
We next define the operator $L_n:H_{A_n}^1(\R,\mathbb{C})\to H_{A_n}^1(\R,\mathbb{C})$ by
\begin{equation}\label{Ln}
\begin{aligned}
L_nw:=&(-\Delta_{A_n}+1)^{-1}W_n[w]\\
=&(-\Delta_{A_n}+1)^{-1}\int_0^1\Big\{\big[|x|^{-1}*|t\tilde{u}_n+(1-t)\tilde{v}_n|^2\big]w\\&
\qquad +2\big[t\tilde{u}_n+(1-t)\tilde{v}_n\big]\Big[|x|^{-1}*\big((t\tilde{u}_n+(1-t)\tilde{v}_n)|w\big)\Big]\Big\}dt,
\end{aligned}
\end{equation}
where we have used the fact that
\begin{equation*}
\begin{aligned}
Df(u)[w]=w(|x|^{-1}*|u|^2) + 2u\big( |x|^{-1}*( u |w)\big).
\end{aligned}
\end{equation*}

The following lemma is concerned with some properties of the operator $L_n$ in $H_{A_n}^1(\R,\mathbb{C})$.

\begin{Lemma}\label{LLn}
The operator $L_n$ defined by \eqref{Ln} is self-adjoint and compact in $H_{A_n}^1(\R,\mathbb{C})$ for all $n\in\N^+$.
\end{Lemma}

\noindent{\bf Proof.} Since the self-adjointness of $L_n$ in $H_{A_n}^1(\R,\mathbb{C})$ is trivial, it suffices to prove the compactness of $L_n$ in $H_{A_n}^1(\R,\mathbb{C})$. Suppose the sequence $\{v_k\}$ is bounded uniformly in $H_{A_n}^1(\R,\mathbb{C})$ for any fixed $n\in \N^+$. Then there exists a function $v\in H_{A_n}^1(\R,\mathbb{C})$ such that up to a subsequence if necessary, $v_k\rightharpoonup v$ weakly in $H_{A_n}^1(\R,\mathbb{C})$ as $k\to \infty$. Define $w_k:=L_nv_k\in H_{A_n}^1(\R,\mathbb{C})$, so that $w_k$ satisfies the equation
\begin{equation}\label{wk}
    -\Delta_{A_n}w_k+w_k=W_n[v_k]\ \text{ \ in \ }\R,
\end{equation}
where the operator $W_n$ is defined by (\ref{3.12M}).
Applying Hardy-Littlewood-Sobolev inequality \eqref{hls}, we deduce from \eqref{tuv} that for all $t\in[0,1]$ and $n\in\N^+$,
\begin{equation}\label{aaa1}
\begin{aligned}
  \||x|^{-1}*|t\tilde{u}_n+(1-t)\tilde{v}_n|^2\|_{6}
   &\leq C\||t\tilde{u}_n+(1-t)\tilde{v}_n|^2\|_{\f65}\\
   &\leq C_1\left(t^2\|\tilde{u}_n\|^2_{\f{12}{5}}+(1-t)^2\|\tilde{v}_n\|^2_{\f{12}{5}}\right),
\end{aligned}
\end{equation}
and
\begin{equation}\label{aa2}
\begin{aligned}
   \||x|^{-1}*\big((t\tilde{u}_n+(1-t)\tilde{v}_n)|v_k\big)\|_{6}
     \leq& C\|\big((t\tilde{u}_n+(1-t)\tilde{v}_n)|v_k\big)\|_{\f{6}{5}}\\
     \leq& C_1 \|t\tilde{u}_n+(1-t)\tilde{v}_n\|_{\f{12}{5}}\|v_k\|_{\f{12}{5}}\\
     \leq& C_2\left(t\|\tilde{u}_n\|_{\f{12}{5}}+(1-t)\|\tilde{v}_n\|_{\f{12}{5}}\right)\|v_k\|_{\f{12}{5}}.
\end{aligned}
\end{equation}
Using Sobolev embedding theorem and H\"older's inequality, we then obtain from \eqref{wk}--\eqref{aa2} that
\begin{equation}\label{bou}
\begin{aligned}
  &\int_{\R}\big(|\nabla_{A_n}w_k|^2+|w_k|^2\big)dx
  =\int_{\R}(w_k|W_n[v_k])dx\\
  =&\int_0^1\int_{\R}\big(|x|^{-1}*|t\tilde{u}_n+(1-t)\tilde{v}_n|^2\big)(w_k|v_k)dxdt\\
  &+2\int_0^1\int_{\R}\big[t\tilde{u}_n+(1-t)\tilde{v}_n|w_k\big]\big[|x|^{-1}*(t\tilde{u}_n+(1-t)\tilde{v}_n|v_k)\big]dxdt\\
  \leq&\int_0^1\||x|^{-1}*|t\tilde{u}_n+(1-t)\tilde{v}_n|^2\|_6\|w_k\|_{\f{12}{5}}\|v_k\|_{\f{12}{5}}dt\\
  &+2\int_0^1\||x|^{-1}*\big((t\tilde{u}_n+(1-t)\tilde{v}_n)|v_k\big)\|_6\|w_k\|_{\f{12}{5}}\|t\tilde{u}_n+(1-t)\tilde{v}_n\|_{\f{12}{5}}dt\\
  \leq& C_1\int_0^1\left(t^2\|\tilde{u}_n\|^2_{\f{12}{5}}+(1-t)^2\|\tilde{v}_n\|^2_{\f{12}{5}}\right)\|w_k\|_{\f{12}{5}}\|v_k\|_{\f{12}{5}}dt\\
  &+C_2\int_0^1\left(t\|\tilde{u}_n\|_{\f{12}{5}}+(1-t)\|\tilde{v}_n\|_{\f{12}{5}}\right)^2\|v_k\|_{\f{12}{5}}\|w_k\|_{\f{12}{5}}dt
  \\\leq& C_3\|w_k\|_{\f{12}{5}}\leq C_4\|w_k\|_{H_{A_n}^1}.
\end{aligned}
\end{equation}
This implies that the sequence $\{w_k\}$ is bounded uniformly in $H_{A_n}^1(\R,\mathbb{C})$, and hence there exists a function  $w\in H_{A_n}^1(\R,\mathbb{C})$ such that up to a subsequence if necessary, $w_k \rightharpoonup w$ weakly in $H_{A_n}^1(\R,\mathbb{C})$ as $k\to\infty$.

We next claim that for any $n\in\N^+$,
\begin{equation}\label{Wto}
    \int_{\R}(w_k|W_n[v_k])dx\to\int_{\R}(w|W_n[v])dx\text{ \ as \ }k\to\infty.
\end{equation}
Actually, we note that
\begin{equation}\label{ww1}
\begin{split}
&\int_{\R}(w_k|W_n[v_k])dx\\
=&\int_0^1\int_{\R}(|x|^{-1}*|h_n(t)|^2)(w_k|v_k)dxdt\\
&+2\int_0^1\int_{\R}(h_n(t)|w_k)\big[|x|^{-1}*(h_n(t)|v_k)\big]dxdt\\
:=&I_k^{1}+I_k^{2},
\end{split}
\end{equation}
where
\begin{equation}\label{hn1}
h_n(t):=t\tilde{u}_n+(1-t)\tilde{v}_n\text{ \ for \ }t\in[0,1].
\end{equation}
We first address the sequence $\{I_k^1\}$ as follows. Since it follows from above that $v_k\to v$ strongly in $L^{\f{12}{5}}_{\loc}(\R)$ and $w_k\to w$ strongly in $L^{\f{12}{5}}_{\loc}(\R)$ as $k\to\infty$, we deduce from H\"older's inequality that for every fixed $R>0$,
\begin{equation}\label{w2}
\begin{split}
  &\left|\int_0^1\int_{B_R}\Big[(|x|^{-1}*|h_n(t)|^2)(w_k|v_k)-(|x|^{-1}*|h_n(t)|^2)(w|v)\Big]dxdt\right|\\
  \leq &C\int_0^1\Big[\||x|^{-1}*|h_n(t)|^2\|_{L^6(B_R)}\big(\|v_k-v\|_{L^{\f{12}{5}}(B_R)}\|w_k\|_{L^{\f{12}{5}}(B_R)}\\
  &+\|v\|_{L^{\f{12}{5}}(B_R)}\|w_k-w\|_{L^{\f{12}{5}}(B_R)}\big)\Big]dt\
  \to 0 \text{ \ as }\ k \to \infty.
\end{split}
\end{equation}

Since $|\tilde{u}_n|\in H^1(\R)$, we also have
\begin{equation*}
    \||x|^{-1}*|\tilde{u}_n|^2\|_6\leq C\|\tilde{u}_n\|^2_{\f{12}{5}}<\infty  \  \text{ for every } \ n\in \N^+.
\end{equation*}
By the absolute continuity, we then obtain that for every $\eta>0$, there exists a sufficiently large constant $R_0>0$ such that for all $R>R_0$,
\begin{equation}\label{un1}
    \||x|^{-1}*|\tilde{u}_n|^2\|_{L^6(B_R^c)}<\eta.
\end{equation}
Similarly,
\begin{equation}\label{vn1}
     \||x|^{-1}*|\tilde{v}_n|^2\|_{L^6(B_R^c)}<\eta.
\end{equation}
Since $|h_n(t)|^2=|t\tilde{u}_n+(1-t)\tilde{v}_n|^2\leq 2\big[t^2|\tilde{u}_n|^2+(1-t)^2|\tilde{v}_n|^2\big]$, we obtain from \eqref{un1} and \eqref{vn1} that
\begin{equation*}
\begin{aligned}
 \||x|^{-1}*|h_n(t)|^2\|_{L^6(B_R^c)}
\leq&2\|t^2|x|^{-1}*|\tilde{u}_n|^2+(1-t)^2|x|^{-1}*|\tilde{v}_n|^2\|_{L^6(B_R^c)}\\
\leq&2t^2\||x|^{-1}*|\tilde{u}_n|^2\|_{L^6(B_R^c)}+2(1-t)^2\||x|^{-1}*|\tilde{v}_n|^2\|_{L^6(B_R^c)}\\
\leq& 2\eta \big[t^2+(1-t)^2\big].
\end{aligned}
\end{equation*}
Consequently, for every $\eta>0$, there exists a constant $R_0>0$ such that for all $R\geq R_0$,
\begin{equation}\label{w3}
\begin{aligned}
 &\sup_{k} \left|\int_0^1\int_{B_R^c}(|x|^{-1}*|h_n(t)|^2)(w_k|v_k)dxdt\right|\\
 \leq& C\sup_k\int_0^1\||x|^{-1}*|h_n(t)|^2\|_{L^6(B_R^c)}\|v_k\|_{L^{\f{12}{5}}(B_R^c)}\|w_k\|_{L^{\f{12}{5}}(B_R^c)}dt\\
 \leq &C_1\sup_k\int_0^1 \||x|^{-1}*|h_n(t)|^2\|_{L^6(B_R^c)}dt< C\eta.
\end{aligned}
\end{equation}
Similarly, we get that for all $R\geq R_0$,
$$\left|\sup_k\int_0^1\int_{B_R^c}(|x|^{-1}*|h_n(t)|^2)(w|v)dxdt\right|<C\eta .$$
Thus, for every $\eta>0$, there exists a large constant $R_0>0$ such that for all $R\geq R_0$,
\begin{equation}\label{w4}
\begin{aligned}
&\limsup_{k\to\infty}\Big|\int_0^1\int_{B_R^c}\Big[(|x|^{-1}*|h_n(t)|^2)(w_k|v_k)-(|x|^{-1}*|h_n(t)|^2)(w|v)\Big]dxdt\Big|
<2C\eta.
\end{aligned}
\end{equation}
We then obtain from \eqref{w2} and \eqref{w4} that
\begin{equation*}
\begin{aligned}
 &\limsup_{k\to \infty}\Big|\int_0^1\int_{\R}\big[(|x|^{-1}*|h_n(t)|^2)(w_k|v_k)-(|x|^{-1}*|h_n(t)|^2)(w|v)\big]dxdt\Big|\\
 \leq& \limsup_{k\to \infty}\Big|\int_0^1\int_{B_R}\big[(|x|^{-1}*|h_n(t)|^2)(w_k|v_k)-(|x|^{-1}*|h_n(t)|^2)(w|v)\big]dxdt\Big|\\
 &+ \limsup_{k\to \infty}\Big|\int_0^1\int_{B_R^c}\big[(|x|^{-1}*|h_n(t)|^2)(w_k|v_k)-(|x|^{-1}*|h_n(t)|^2)(w|v)\big]dxdt\Big|\\\leq&0+2C\eta=2C\eta.
\end{aligned}
\end{equation*}
Since $\eta>0$ is arbitrary, it follows from above that
\begin{equation}\label{Ik1}
    \Big|\int_0^1\int_{\R}\big[(|x|^{-1}*|h_n(t)|^2)(w_k|v_k)-(|x|^{-1}*|h_n(t)|^2)(w|v)\big]dxdt\Big|\to 0 \ \ \mbox{as}\ \, k\to\infty,
\end{equation}
which therefore gives the convergence of the sequence $\{I_k^1\}$ as $k\to\infty$.

On the other hand, we derive from  H\"older's inequality that
\begin{equation*}
\begin{aligned}
 &\left|\int_0^1\int_{\R}\Big\{\big[|x|^{-1}*(h_n(t)|v_k)\big](h_n(t)|w_k)-\big[|x|^{-1}*(h_n(t)|v)\big](h_n(t)|w)\Big\}dxdt\right|\\
 \leq& \Big|\int_0^1\int_{\R}\Big[|x|^{-1}*(h_n(t)|v_k)-|x|^{-1}*(h_n(t)|v)\Big](h_n(t)|w_k)dxdt\Big|\\
 &+\Big|\int_0^1\int_{\R}|x|^{-1}*(h_n(t)|v)\left[(h_n(t)|w_k)-(h_n(t)|w)\right]dxdt\Big|\\
 \leq&\int_0^1\big\||x|^{-1}*(h_n(t)|v_k)-|x|^{-1}*(h_n(t)|v)\big\|_6\|h_n(t)\|_{\f{12}{5}}\|w_k\|_{\f{12}{5}}dt\\
 &+\Big|\int_0^1\||x|^{-1}*(h_n(t)|v)\|_6\|(h_n(t)|w_k)-(h_n(t)|w)\big\|_{\f65}dt\Big|.
\end{aligned}
\end{equation*}
For any fixed $n\in\N^+$, since the sequence  $\{h_n(t)\}$ is uniformly
bounded in $L^{12/5}(\R)$ with respect to $t\in[0,1]$, and $\|h_n(t)\|_{L^{12/5}(B_R^c)}\to0$ uniformly in $t\in[0,1]$   as $R\to\infty$, the same argument of estimating $I_k^1$ yields that
\begin{equation*}
\|(h_n(t)|w_k)-(h_n(t)|w)\|_{\f65}\to 0 \ \text{ as } \ k\to \infty,
\end{equation*}
and
\begin{equation*}
\||x|^{-1}*(h_n(t)|v_k)-|x|^{-1}*(h_n(t)|v)\|_6\to 0 \ \text{ as } \  k\to \infty.
\end{equation*}
We thus have
\begin{equation}\label{w6}
 \begin{split}   \Big|\int_0^1\int_{\R}&\Big\{\big[|x|^{-1}*(h_n(t)|v_k)\big](h_n(t)|w_k)\\
 &-\big[|x|^{-1}*(h_n(t)|v)\big](h_n(t)|w)\Big\}dxdt\Big|\to 0\ \ \mbox{as}\ \, k\to\infty,
 \end{split}
\end{equation}
which therefore gives the convergence of $\{I_k^2\}$ as $k\to\infty$.

We now conclude from \eqref{Ik1} and (\ref{w6}) that the claim \eqref{Wto} holds true.
Multiplying the equation \eqref{wk} by $\bar {w}$ and integrating over $\R$, we then obtain from \eqref{Wto} that
\begin{equation*}
\begin{aligned}
    \int_{\R}\big(|\nabla_{A_n}w|^2+|w|^2\big)dx=& \lim_{k\to \infty}(w_k|w)_{H_{A_n}^1}=\lim_{k\to \infty}\int_{\R}(w|W_n[v_k])dx\\
    =&\lim_{k\to \infty}\int_{\R}(w_k|W_n[v_k])dx\\
    =&\lim_{k\to \infty}\int_{\R}\big(|\nabla_{A_n}w_k|^2+|w_k|^2\big)dx.
\end{aligned}
\end{equation*}
This implies that $w_k \to w$ strongly in $H_{A_n}^1(\R,\mathbb{C})$ as $k\to\infty$, which therefore completes the proof of Lemma \ref{LLn}.\qed

\subsection{Proof of Theorem \ref{T1.1} }

The main aim of this subsection is to complete the proof of Theorem \ref{T1.1}.  Towards this aim, let $\la_k(L_n)$ be the $k$-th eigenvalue of $L_n$ defined by \eqref{Ln} in $H_{A_n}^1(\R,\mathbb{C})$, where $k\in\N^+$. We note for any $v\in H_{A_n}^1(\R,\mathbb{C})$,
\begin{equation}\label{p}
\begin{aligned}
&\int_{\R}(v|W_n[v])dx\\
=&\int_0^1\int_{\R}\Big[
(|x|^{-1}*|h_n(t)|^2)|v|^2
+2(h_n(t)|v)\big(|x|^{-1}*(h_n(t)|v)\big)
\Big]dxdt \geq 0.
\end{aligned}
\end{equation}
We then deduce from Lemma \ref{LLn} and \eqref{p} that the eigenvalues of $L_n$ in $H_{A_n}^1(\R,\mathbb{C})$ satisfy
\begin{equation*}
\begin{aligned}
0\leq\cdots\leq \la_{k+1}(L_n)\leq\la_k(L_n)\leq \cdots\leq \la_1(L_n),
\end{aligned}
\end{equation*}
and
\begin{equation*}
\begin{aligned}
\lim_{k\to\infty} \la_k(L_n)=0 \text{ \ for any\ } n\in \N^+.\\
\end{aligned}
\end{equation*}
By Fischer's min-max principle (cf. \cite[Theorem 28.4]{Lax}), we also have
\begin{equation}\label{lak}
\begin{aligned}
\la_k(L_n)=\sup_{\substack{{E\subset H_{A_n}^1(\R,\mathbb{C})} \\ \dim E=k}}\inf_{v\in E\setminus\{0\}}\f{\int_{\R}(v|W_n[v])dx}{\int_{\R}\big(|\nabla_{A_n}v|^2+|v|^2\big)dx}.
\end{aligned}
\end{equation}
We also define the operator $L: H^1(\R,\mathbb{C})\to H^1(\R,\mathbb{C})$   by
\begin{equation}\label{lakA}
Lv=(-\Delta+1)^{-1}W[v],
\end{equation}
where $W[v]:=(|x|^{-1}*|u_0|^2)v+2u_0(|x|^{-1}*(u_0|v))$, and $u_0>0$ is the unique real-valued ground state of \eqref{w}.

For convenience we denote $(\la_k(L_n), w_n^k)$ and $(\la_k(L), w^k)$ to be the $k$-th eigenpair of the operators $L_n$ and $L$, respectively.
We now establish the following proposition on the convergence of $\lambda_k(L_n)$ in $H_{A_n}^1(\R,\mathbb{C})$ as $n\to\infty$.

\begin{Proposition}\label{P4.1}
Suppose $b_n\to 0$ as $n\to \infty$, then we have
\begin{equation*}
\la_k(L_n)\to \la_k(L) \ \text{ as } \ n\to \infty.
\end{equation*}
Moreover, if $\la_k(L)>0$ and $ w_n^k\in H_{A_n}^1(\R,\mathbb{C})$ satisfies
\begin{equation*}
    L_n w_n^k =\la_k(L_n) w_n^k  \ \text{ in } \ \R, \ \ \int_{\R}\big(|\nabla_{A_n} w_n^k |^2+| w_n^k |^2\big)dx=1,
\end{equation*}
then  up to a subsequence if necessary,
\begin{equation}\label{4:31A}
 w_n^k \to w^k \,\   and  \,\ \nabla_{A_n}{w_n^k}\to \nabla w^k \ \text{ strongly in }\, L^2(\R) \, \text{ as } \, n\to \infty.
\end{equation}
\end{Proposition}

\noindent{\bf Proof.} Since it yields from \cite[Theorem 7.22]{Lieb01} that the space $C_c^1(\R,\mathbb{C})$ is dense in $H_{A_n}^1(\R,\mathbb{C})$, we deduce  from \eqref{Ln} and \eqref{lak} that
\begin{equation*}
\begin{aligned}
&\la_k(L_n)\\
=&\sup_{\substack{{E\subset C_c^1(\R,\mathbb{C})} \\ \dim E=k}}\inf_{v\in E\setminus\{0\}}\f{\int_0^1\int_{\R}\Big[(|x|^{-1}*|h_n|^2)|v|^2+2(h_n|v)\big(|x|^{-1}*(h_n|v)\big)\Big]dxdt}{\int_{\R}\big(|\nabla_{A_n}v|^2+|v|^2\big)dx}\\
:=&\sup_{\substack{{E\subset C_c^1(\R,\mathbb{C})} \\ \dim E=k}}\inf_{v\in E\setminus\{0\}}R_n(v),
\end{aligned}
\end{equation*}
and
\begin{equation}\label{llak}
\begin{aligned}
 \la_k(L)&=\sup_{\substack{{E\subset C_c^1(\R,\mathbb{C})} \\ \dim E=k}}\inf_{v\in E\setminus\{0\}}\f{\int_{\R}\Big[(|x|^{-1}*|u_0|^2)|v|^2+2(u_0|v)\big(|x|^{-1}*(u_0|v)\big)\Big]dx}{\int_{\R}\big(|\nabla v|^2+|v|^2\big)dx}\\
 &:=\sup_{\substack{{E\subset C_c^1(\R,\mathbb{C})} \\ \dim E=k}}\inf_{v\in E\setminus\{0\}}R(v),
\end{aligned}
\end{equation}
where $h_n:=h_n(t)=t\tilde{u}_n+(1-t)\tilde{v}_n$ for $t\in [0,1]$.
Since $\nabla_{A_n}u=\nabla u+iA_nu$ and $b_n \to 0$ as $n\to\infty$, we obtain that for any $v\in C_c^1(\R,\mathbb{C})$,
\begin{equation*}
\begin{aligned}
     \|\nabla_{A_n}v-\nabla  v\|_{L^2(\R)}&=\|iA_nv\|_{L^2(\operatorname{supp}(v))}\\&\leq \|A_n\|_{L^2(\operatorname{supp}(v))}\|v\|_{L^{\infty}}\to 0  \  \text{ as } n \to \infty.
\end{aligned}
\end{equation*}
Therefore, if the dimension of the linear subspace $E\subset C_c^1(\R,\mathbb{C})$ is $k$, then it follows from Lemma \ref{u1} and above that  $R_n(v)\to R(v)$ uniformly in $v\in E\setminus\{0\}$ as $n\to \infty$. Hence for any $\varepsilon>0$, there exists $N_{\varepsilon}\in \mathbb{N}^+$ such that for any $n>N_{\varepsilon}$,
\begin{equation*}
    R_n(v)\geq R(v)-\varepsilon \ \ \text{ for all } \ v\in E\setminus\{0\}.
\end{equation*}
By the definition of $\la_k(L_n)$, then for any fixed finite dimensional subspace $E\subset C_c^1(\R,\mathbb{C})$,
\begin{equation}\label{R22}
    \la_k(L_n)\geq \inf_{v\in E\setminus\{0\}}R_n(v)\geq\inf_{v\in E\setminus\{0\}}R(v)-\varepsilon \ \ \text{for all } \ n>N_{\varepsilon}.
\end{equation}
This thus implies that
\begin{equation*}
     \liminf_{n\to \infty}\la_k(L_n)\geq\liminf_{n\to\infty}\lim_{\varepsilon\to0^+}\inf_{v\in E\setminus\{0\}}R(v)-\varepsilon=\inf_{v\in E\setminus\{0\}}R(v).
\end{equation*}
Since $E$ is an arbitrary $k$-dimensional subspace of $C_c^1(\R,\mathbb{C})$, we obtain from above that
\begin{equation}\label{b}
    \liminf_{n\to \infty}\la_k(L_n)\geq\sup_{\substack{{E\subset C_c^1(\R,\mathbb{C})} \\ \dim E=k}}\inf_{v\in E\setminus\{0\}}R(v)= \la_k(L).
\end{equation}

Similar to \eqref{p}, the quadratic form of $L$ in $H^1(\R,\mathbb{C})$ is nonnegative, which gives that $\lambda_k(L)\ge0$ holds for all $k\in \mathbb{N}^+$. As a consequence, if
\begin{equation*}
    \la_k^*:=\limsup_{n\to \infty}\la_k(L_n)=0,
\end{equation*}
 then we immediately obtain from \eqref{b} that up to a subsequence if necessary,
\begin{equation*}
    \lim_{n\to \infty}\la_k(L_n)=\la_k(L).
\end{equation*}
It remains to consider the case $\la^*_k>0$. Since $0\leq \la_k(L_n)\leq\la_1(L_n)$ is bounded for each $k\in\mathbb{N}^+$, we obtain that up to a subsequence if necessary,
\begin{equation}\label{3.235}
    \lim_{n\to \infty}\la_k(L_n)= \la_k^*>0.
\end{equation}
Let $(\la_k(L_n), w^k_n)$ be the $k$-th eigenpair of $L_n$ in $H_{A_n}^1(\R,\mathbb{C})$, so that $\int_{\R}(|\nabla_{A_n}w_n^k|^2+|w_n^k|^2)dx=1$, and $L_nw_n^k=\la_k(L_n)w_n^k$ in $\R$. We then derive from \eqref{Ln} that
\begin{equation}\label{3.225}
\int_{\R}(w_n^k|W_n[w_n^k])dx=\la_k(L_n),
\end{equation}
and
\begin{equation}\label{323}
    \int_{\R}(w_n^k|W_n[w_n^j])dx=\la_j(L_n)\int_{\R}\big[(\nabla_{A_n}w_n^k|\nabla_{A_n}w_n^j)+(w_n^k|w_n^j)\big]dx=\la_j(L_n)\delta_{kj},
\end{equation}
where the operator $W_n$ is as in \eqref{3.12M}.
Applying \cite[Lemma 2.2]{DB}, we then derive from \eqref{323} that there exists a function $w^k \in H^1(\R,\mathbb{C})$ such that up to a subsequence if necessary, $w_n^k\rightharpoonup w^k$ and $\nabla_{A_n}w_n^k\rightharpoonup \nabla w^k$ weakly in $L^2(\R)$ as $n\to \infty$.

Similar to \eqref{Wto}, we obtain from \eqref{3.225} that
\begin{equation}\label{3.245}
    0<\la_k^*=\limsup_{n\to\infty}\int_{\R}(w_n^k|W_n[w_n^k])dx= \int_{\R}(w^k|W[w^k])dx,
\end{equation}
which then implies that $w^k\not\equiv 0$. We then conclude from \eqref{b}--\eqref{3.225} and \eqref{3.245} that
\begin{equation}\label{20}
\begin{aligned}
\la_k(L)\leq&\liminf_{n\to \infty}\la_k(L_n)\leq\la_k^*=\limsup_{n\to \infty}\int_{\R}(w_n^k|W_n[w_n^k])dx\\
=&\int_{\R}(w^k|W[w^k])dx\leq \f{\int_{\R}(w^k|W[w^k])dx}{\int_{\R}\big(|\nabla w^k|^2+|w^k|^2\big)dx},
\end{aligned}
\end{equation}
where we have used the estimate $\|w^k\|_{H^1}\leq1$ for the last inequality. Moreover,   similar to \eqref{Wto}, one can deduce from \eqref{323} that $\int_{\R}(w^k|W[w^j])=\lim_{n\to\infty}(w_n^k|W_n[w_n^j])=0$ holds for $j\ne k$.

We are next ready to finish the proof by induction. For $k=1$, we obtain from \eqref{llak} that
\begin{equation}\label{d}
\f{\int_{\R}(w^1|W[w^1])dx}{\int_{\R}\big(|\nabla w^1|^2+|w^1|^2\big)dx}\leq\sup_{w\in H^1(\R,\mathbb{C})\setminus\{0\}}\f{\int_{\R}(w|W[w])dx}{\int_{\R}\big(|\nabla w|^2+|w|^2\big)dx}=\la_1(L).
\end{equation}
It then follows from \eqref{3.245}--\eqref{d} that
$$Lw^1=\la_1(L)w^1 \ \  \mbox{in}\ \  \R, \ \ \mbox{and} \ \ \la_1^*=\la_1(L).$$
Hence, we deduce from \eqref{20} and above that
\begin{equation*}
    \la_1(L)\geq \la_1(L)\int_{\R}\big(|\nabla  w^1|^2+|w^1|^2\big)dx=\int_{\R}(w^1|W[w^1])dx\geq\liminf_{n\to \infty}\la_1(L_n)\geq\la_1(L).
\end{equation*}
We thus have $\nabla_{A_n}w_n^1\to \nabla  w^1$ and $w_n^1\to w^1$ strongly in $L^2(\R)$ as $n\to \infty$, which proves (\ref{4:31A}) for $k=1$.

For any $k\geq1$,  suppose that (\ref{4:31A}) holds for all $j\in\{1,\cdots,k-1\}$. We then have
\begin{equation*}
    (w_n^k,w_n^j)_{H_{A_n}^1}=0, \ \ j=1,\cdots,k-1,
\end{equation*}
which further implies that
\begin{equation*}
    (w^k,w^j)_{H^1}=0, \ \ j=1,\cdots,k-1.
\end{equation*}
Since (\ref{4:31A}) holds for all $j\in\{1,\cdots,k-1\}$, it follows from \eqref{llak} that
\begin{equation}\label{g}
    \f{\int_{\R}(w^k|W[w^k])dx}{\int_{\R}(|\nabla w^k|^2+|w^k|^2)dx}\leq\la_k(L).
\end{equation}
We thus obtain from \eqref{20} and \eqref{g} that $Lw^k=\la_k(L)w^k$ in $\R$ and $\la_k^*=\la_k(L)$. Hence, we have
\begin{equation*}
     \la_k(L)\geq \la_k(L)\int_{\R}\big(|\nabla  w^k|^2+|w^k|^2\big)dx=\int_{\R}(w^k|W[w^k])dx\geq\liminf_{n\to \infty}\la_k(L_n)\geq\la_k(L).
\end{equation*}
This implies that $w_n^k\to w^k$ and $\nabla_{A_n}w_n^k\to \nabla  w^k$ strongly in $L^2(\R)$ as $n\to \infty$, $i.e.,$ (\ref{4:31A}) holds for any $k\geq1$. The proof of Proposition \ref{P4.1} is therefore complete.\qed

\vspace{5pt}
Applying Proposition \ref{P4.1}, we are next able to address the proof of Theorem \ref{T1.1}.

\vspace{5pt}
\noindent{\bf Proof of Theorem \ref{T1.1}.} Since it follows from Lemma \ref{LLn} that the operator $L_n$ defined by \eqref{Ln} is self-adjoint and compact in $H_{A_n}^1(\R,\mathbb{C})$, we derive from spectral theory that
\begin{equation*}
    L_n=\sum_{k=1}^{\infty} \la_k(L_n)|w_n^k\rangle \langle w_n^k|,
\end{equation*}
where   $0\leq \cdots\leq \la_k(L_n)\leq \cdots\leq \la_1(L_n)$ satisfy $\lim_{k\to \infty} \la_k(L_n)=0$ for all $n\in \N^+$, and $(\la_k(L_n), w_n^k)$ is the $k$-th eigenpair of the operator $L_n$  in $H_{A_n}^1(\R,\mathbb{C})$. Set
\begin{equation*}
    P_n^1=|w_n^1\rangle \langle w_n^1|, \ P_n^2=\sum_{k=2}^5|w_n^k\rangle \langle w_n^k| \,\ \text{ and }\ P_n^3=I-P_n^1-P_n^2,
\end{equation*}
so that
\begin{equation*}
     L_nP_n^1=P_n^1L_n,\quad L_nP_n^3=P_n^3L_n.
\end{equation*}
It then follows from \eqref{3.12} that $L_n(\widetilde{u}_n-\widetilde{v}_n)=\widetilde{u}_n-\widetilde{v}_n$ in $\R$, where $\widetilde{u}_n$ and $\widetilde{v}_n$ given by (\ref{tuv}) are two distinct solutions of \eqref{equ}. This thus yields that
\begin{equation}\label{Pn1}
\begin{aligned}
 \|P_n^1(\widetilde{u}_n-\widetilde{v}_n)\|^2_{H_{A_n}^1}&=\big(P_n^1(\widetilde{u}_n-\widetilde{v}_n)|P_n^1L_n
 (\widetilde{u}_n-\widetilde{v}_n)\big)_{H_{A_n}^1}\\&=\la_1(L_n)\|P_n^1(\widetilde{u}_n-\widetilde{v}_n)\|_{H_{A_n}^1}^2.
\end{aligned}
\end{equation}
Applying \cite{Len} and \cite[Proposition~5]{FLR13}, it then follows from \eqref{12} that
\begin{equation*}
    \la_1(L)=3;\ \la_k(L)=1,\ \ k=2,3,4,5; \ \ \la_k(L)<1, \ \ k\geq 6.
\end{equation*}
Applying Proposition \ref{P4.1}, we then get that as $n\to\infty$,
\begin{equation*}
\begin{aligned}
&\la_1(L_n)\to \la_1(L)>1; \ \la_k(L_n)\to \la_k(L)=1, \ \ 2\leq k\leq 5;\ \  \la_6(L_n)\to \la_6(L)<1.
\end{aligned}
\end{equation*}
Since $\la_1(L_n)\to \la_1(L)>1$ as $n\to \infty$, it follows from \eqref{Pn1} that $P_n^1(\widetilde{u}_n-\widetilde{v}_n)=0$ holds for sufficiently large $n\in \N^+$. Similarly,
\begin{equation*}
\begin{aligned}
 \|P_n^3(\widetilde{u}_n-\widetilde{v}_n)\|^2_{H_{A_n}^1}&=\big(P_n^3(\widetilde{u}_n-\widetilde{v}_n)|P_n^3L_n(\widetilde{u}_n-\widetilde{v}_n)\big)_{H_{A_n}^1}\\&\leq\la_6(L_n)\|P_n^3(\widetilde{u}_n-\widetilde{v}_n)\|_{H_{A_n}^1}^2.
\end{aligned}
\end{equation*}
Since $\la_6(L_n)\to \la_6(L)<1$ as $n\to \infty$, the above identity holds, if and only if $P_n^3(\widetilde{u}_n-\widetilde{v}_n)=0$ holds for sufficiently large $n\in \N^+$. Hence,   $\widetilde{u}_n-\widetilde{v}_n\in \operatorname{span}_{\mathbb{R}}\{w_n^2, w_n^3, w_n^4, w_n^5\}$ holds for   sufficiently large $n\in \N^+$.

By contradiction, we now suppose that $\widetilde{u}_n\not \equiv\widetilde{v}_n$ holds for sufficiently large $n\in \N^+$. Define the $H_{A_n}^1$-normalized function
\begin{equation}\label{r}
    r_n=\f{\widetilde{u}_n-\widetilde{v}_n}{\|\widetilde{u}_n-\widetilde{v}_n\|}_{H_{A_n}^1} \ \text{ for sufficiently large } n\in \N^+,
\end{equation}
so that $r_n\in \operatorname{span}\{w_n^2,w_n^3,w_n^4,w_n^5\}$. It then yields from Proposition \ref{P4.1} that there exist $y\in \R$ and $\la\in \mathbb{R}$ such that
\begin{equation*}
\begin{aligned}
    &r_n\to r:=\nabla u_0\cdot y+\la iu_0 ,\ \nabla_{A_n}r_n\to \nabla r \ \text{ strongly in } L^2(\R) \ \text{ as } \ n\to\infty ,
\end{aligned}
\end{equation*}
$i.e.$,
\begin{equation}\label{z1}
    \lim_{n\to\infty}\int_{\R}\big(|\nabla_{A_n}r_n-\nabla r|^2+|r_n-r|^2\big)dx=0.
\end{equation}
Moreover,  we deduce from Lemma \ref{u1} that for sufficiently large $n\in\N^+$,
\begin{equation}\label{z2}
   \int_{\R}\big[(\nabla_{A_n}r_n|\nabla r)+(r_n|r)\big]dx=0.
\end{equation}
we then obtain from \eqref{r}--\eqref{z2} that
\begin{equation*}
\begin{aligned}
 \|\widetilde{u}_n-\widetilde{v}_n\|^2_{H_{A_n}^1}&= \|\widetilde{u}_n-\widetilde{v}_n\|^2_{H_{A_n}^1}\int_{\R}\big[(\nabla_{A_n}r_n|\nabla_{A_n} r_n)+(r_n|r_n)\big]dx\\&=\|\widetilde{u}_n-\widetilde{v}_n\|^2_{H_{A_n}^1}\int_{\R}\big[(\nabla_{A_n}r_n-\nabla r|\nabla_{A_n} r_n)+(r_n-r|r_n)\big]dx\\&\leq\|\widetilde{u}_n-\widetilde{v}_n\|^2_{H_{A_n}^1}\sqrt{\int_{\R}\big(|\nabla_{A_n}r_n-\nabla r|^2+|r_n-r|^2\big)dx}.
\end{aligned}
\end{equation*}
Since $\|\widetilde{u}_n-\widetilde{v}_n\|^2_{H_{A_n}^1}\ne0$, the above estimate gives that
\begin{equation*}
   \int_{\R}\big(|\nabla_{A_n}r_n-\nabla r|^2+|r_n-r|^2\big)dx\geq 1,
\end{equation*}
which however contradicts with \eqref{z1}. We thus conclude that $\widetilde{u}_n\equiv \widetilde{v}_n$ holds for sufficiently large $n\in\N^+$, which proves the uniqueness (\ref{1A:T1.1}).

It remains to prove the nondegeneracy \eqref{1B:T1.1}. Let $u_n$ be any ground state of \eqref{equ} satisfying
\begin{equation*}
 A_n=\frac{b_n}{2}(-x_2,x_1,0),\ \ \lim_{n\to\infty}b_n=0.
\end{equation*}
Define the linearized operator at $u_n$ by
\begin{equation*}
\begin{aligned}
\mathcal L_n w
:={}&-\Delta_{A_n}w+w
-\big(|x|^{-1}*|u_n|^2\big)w\\
&-2u_n\big(|x|^{-1}*(u_n|w)\big)\text{ \ in \ }\R.
\end{aligned}
\end{equation*}
By the phase and magnetic translation invariance of \eqref{equ}, the following
four functions
\begin{equation*}
iu_n,\quad
Z_{n,1}:=-\partial_{x_1}u_n-\frac{ib_n}{2}x_2u_n,\quad
Z_{n,2}:=-\partial_{x_2}u_n+\frac{ib_n}{2}x_1u_n,\quad
Z_{n,3}:=-\partial_{x_3}u_n
\end{equation*}
satisfy
\begin{equation}\label{kerl}
iu_n,\ Z_{n,1},\ Z_{n,2},\ Z_{n,3}\in\ker\mathcal L_n.
\end{equation}
Moreover, these four functions are linearly independent over
$\mathbb R$.

Denote
\begin{equation*}
K_nw:=(-\Delta_{A_n}+1)^{-1}\Big[\big(|x|^{-1}*|u_n|^2\big)w+2u_n\big(|x|^{-1}*(u_n|w)\big)\Big],
\end{equation*}
so that
\begin{equation}\label{kerlk}
\ker\mathcal L_n=\ker(I-K_n).
\end{equation}
Since magnetic translations and phase shifts preserve the spectrum of
$K_n$, together with Lemma~\ref{u1}, the same spectral convergence argument as in
Proposition~\ref{P4.1} gives that as $n\to\infty$,
\begin{equation*}
\lambda_1(K_n)\to\lambda_1(L)=3,\ \
\lambda_k(K_n)\to\lambda_k(L)=1\ (2\leq k\leq5),\ \
\lambda_6(K_n)\to\lambda_6(L)<1,
\end{equation*}
where the operator $L$ is as in Proposition~\ref{P4.1}.
Hence, for all sufficiently large $n>0$,
\begin{equation*}
\dim\ker(I-K_n)\leq4.
\end{equation*}
Together with \eqref{kerl} and \eqref{kerlk}, this implies that
\begin{equation*}
\ker\mathcal L_n
=
\operatorname{span}_{\mathbb R}
\{iu_n,Z_{n,1},Z_{n,2},Z_{n,3}\}.
\end{equation*}
Hence, \eqref{1B:T1.1} holds for all sufficiently small $|b|>0$,
which therefore completes the proof of Theorem~\ref{T1.1}.  \qed

\section{Symmetry and Monotonicity of Ground States}\label{S4}
In this section we mainly prove Theorem~\ref{T1.2}, which is concerned with the symmetry and monotonicity of ground states for \eqref{equ}, where $|b|>0$ is sufficiently small. We first consider the following auxiliary equation
\begin{equation}\label{22}
	-\Delta u+(1+|A(x)|^2)u=u(|x|^{-1}*|u|^2) \text{ \ in } \ \R, \ \  u\in \H_b,
\end{equation}
where $A(x)=\frac{b}{2}(-x_2,x_1,0)$ and the space $ \H_b$ is defined by
\begin{equation}\label{Hb}
     \H_b=\big\{ u \in H^1(\R, \mathbb{C}) : \,  \int_{\R}|A(x)|^2|u|^2dx<\infty \big\}.
\end{equation}
The  ground state energy of \eqref{22} is defined by
\begin{equation}\label{web}
 \widetilde{e}(b)=\inf\Big\{\widetilde{I}_b(u): \ \, u\in\H_b\setminus\{0\},\, \langle\widetilde{I}'_b(u),\,u\rangle=0\Big\},
\end{equation}
where the energy functional $\widetilde{I}_{b}(u)$ satisfies
\begin{equation*}
\widetilde{I}_{b}(u)=\f{1}{2}\int_{\R}\big[|\nabla u|^2+(1+|A(x)|^2)|u|^2\big]dx-\f14\iint_{\R\times \R}\f{|u(x)|^2|u(y)|^2}{|x-y|}dxdy.
\end{equation*}
A function $u\in\H_b$ is called a ground state of \eqref{22}, if $u$ is a weak solution of \eqref{22} and satisfies $ \widetilde{e}(b)= \widetilde{I}_{b}(u)$.

The following lemma gives the existence of ground states for \eqref{22}.

\begin{Lemma}\label{L4.1}
Suppose $A=\f{b}{2}(-x_2,x_1,0)$, then \eqref{22} admits a ground state $u\in\H_b$, in the sense that
\begin{equation*}
\widetilde{I}_b(u)=\widetilde{e}(b), \ \widetilde{I}_b'(u)=0.
\end{equation*}
Moreover,
\begin{equation}\label{eb}
    \widetilde{e}(b)=\f14\inf_{u\in\H_b\setminus\{0\}}\widetilde{Q}_b(u)^2,
\end{equation}
where the functional $\widetilde{Q}_b:\H_b\setminus\{0\}\to \mathbb{R}$ is defined by
\begin{equation}\label{wideQ}
    \widetilde{Q}_b(u)=\f{\int_{\R}\big[|\nabla u|^2+(1+|A|^2)|u|^2\big]dx}{(\iint_{\R\times \R}\f{|u(x)|^2|u(y)|^2}{|x-y|}dxdy)^{\f12}}.
\end{equation}
\end{Lemma}

Since the proof of Lemma  \ref{L4.1} is similar to that of Proposition~\ref{ex}, we omit it for simplicity. We should point out that the proof of Lemma  \ref{L4.1} also yields the existence of minimizers for the following minimization
\begin{equation}\label{qb11}
\widetilde q(b):=\inf_{u\in\H_b\setminus\{0\}}\widetilde Q_b(u),
\end{equation}
where the functional $\tilde{Q}_b(u)$ is as in \eqref{wideQ}.
We then establish the following proposition.

\begin{Proposition}\label{P}
Let $e(b)$ and $\widetilde{e}(b)$ be defined by \eqref{inf} and \eqref{web}, respectively, where $A(x)=\frac{b}{2}(-x_2,x_1,0)$. Then there exists a sufficiently small $\varepsilon>0$ such that   for any $0<|b|<\varepsilon$,
\begin{equation}\label{equal-energy}
e(b)=\widetilde{e}(b),
\end{equation}
and we have the following conclusions:
\begin{enumerate}
\item If $u$ is a ground state of \eqref{equ}, then there exists $x_0\in\R$ such that $\tau_{x_0,A}u$ is a ground state of \eqref{22}.
\item If $w$ is a ground state of \eqref{22}, then $w$ is also a ground state of \eqref{equ}. Moreover, there exists a constant $\theta\in\mathbb R$ such that $w\equiv e^{i\theta}|w|$.
\end{enumerate}
\end{Proposition}

\noindent{\bf Proof.}
We first prove \eqref{equal-energy} as follows. Let $u$ be a ground state of \eqref{equ}. By Kato's inequality
(cf. \cite{Lieb01}), one can prove that $|u|$ decays exponentially at
infinity. Hence, the $L^2$-mass center
\begin{equation}\label{a}
	x_0=-\f{\int_{\R}x|u(x)|^2dx}
	{\int_{\R}|u(x)|^2dx}
\end{equation}
is well-defined. Set $v:=\tau_{x_0,A}u$, so that $v$ is also a
ground state of \eqref{equ}. Moreover, it follows from \eqref{1.45}
and \eqref{a} that
\begin{equation}\label{xx}
	\int_{\R}x|v(x)|^2dx=0.
\end{equation}
By a similar argument of \cite[Proposition 5.1 and Lemma 5.2]{DB}, we then obtain from \eqref{xx} that there exists a sufficiently small $\varepsilon>0$ such that for any $0<|b|<\varepsilon$,
\begin{equation*}
	iA(x)\cdot\nabla v=0 \ \text{ in } \ \R,
\end{equation*}
and  $v$ is also a solution of \eqref{22}. It then follows from \eqref{inf} and \eqref{web} that
\begin{equation}\label{energy1}
\widetilde{e}(b)\leq\widetilde{I}_b(v)=I_b(v)=e(b).
\end{equation}

On the other hand, it yields from Lemma \ref{L4.1} that \eqref{22} admits a
ground state $\widetilde v_b$, which can be chosen to be real-valued
and nonnegative. We then have
\begin{equation*}
|\nabla_A\widetilde v_b|^2=|\nabla\widetilde v_b|^2+|A|^2|\widetilde v_b|^2 \ \text{ in }\ \R.
\end{equation*}
Hence, the Nehari constraints associated with $I_b$ and
$\widetilde I_b$ coincide at $\widetilde v_b$, and
\begin{equation*}
\widetilde e(b)=\widetilde I_b(\widetilde v_b)=I_b(\widetilde v_b)\geq e(b).
\end{equation*}
Together with \eqref{energy1}, this thus proves \eqref{equal-energy} for $0<|b|<\varepsilon$.

It now follows from \eqref{equal-energy} and \eqref{energy1} that for $0<|b|<\varepsilon$,
\begin{equation*}
\widetilde I_b(v)=\widetilde e(b),
\end{equation*}
and hence $v=\tau_{x_0,A}u$ is a ground state of \eqref{22}, where $u$ is a ground state of \eqref{equ}. This proves Proposition \ref{P}(1).

To prove Proposition \ref{P}(2), let $w$ be a ground state
of \eqref{22}. It then follows from \cite[Theorem 7.8]{Lieb01} that
\begin{equation*}
\widetilde Q_b(|w|)\leq\widetilde Q_b(w).
\end{equation*}
We thus derive from Lemma \ref{L4.1} that a suitable multiplicity of $|w|$ is a nonnegative ground state of \eqref{22}.
Hence, by the strong maximum principle, we deduce that $|w|>0$ in $\R$. This further yields from \cite[Theorem 7.8]{Lieb01} that there exists $\theta\in\mathbb R$ such that $w=e^{i\theta}|w|$. We then have
\begin{equation*}
	|\nabla_Aw|^2=|\nabla w|^2+|A|^2|w|^2 \ \text{ in } \ \R.
\end{equation*}
Hence, the Nehari constraint of $I_b$   coincides with the Nehari constraint of
$\widetilde I_b$ at $w$, and we obtain from \eqref{equal-energy} that for $0<|b|<\varepsilon$,
\begin{equation*}
	I_b(w)=\widetilde I_b(w)=\widetilde e(b)=e(b).
\end{equation*}
Further, since the Nehari manifold of $I_b$ is a natural
constraint, we have
\begin{equation*}
	I_b'(w)=0.
\end{equation*}
Thus, $w$ is a ground state of \eqref{equ}. This therefore completes the
proof of Proposition \ref{P}. \qed

\vspace{5pt}

We next address the uniqueness and symmetry of positive ground states
of \eqref{22}.

\begin{Lemma}\label{P5.1}
There exists a sufficiently small $\varepsilon>0$ such that for any $0<|b|<\varepsilon$, the problem \eqref{22} admits a positive
real-valued ground state $U_b$, which is unique, up to translations in
the $x_3$-direction. Moreover, $U_b$ can be chosen to be radially  symmetric and decreasing in $(x_1,x_2)$ and  $x_3$.
\end{Lemma}

\noindent{\bf Proof.}
Let $v\in\H_b\setminus\{0\}$ be a minimizer of $\widetilde q(b)$ defined by \eqref{qb11}. Choose $t_v>0$ such that $t_vv$ satisfies the Nehari constraint of $\widetilde I_b$. By Lemma \ref{L4.1}, $t_vv$ is a ground state of \eqref{22}. By Proposition \ref{P}, then there exists a sufficiently small $\varepsilon>0$ such that for any $0<|b|<\varepsilon$, there exists $\theta\in\mathbb R$ such that $t_vv\equiv e^{i\theta}|t_vv|$,
and  $t_vv$ is also a ground state of \eqref{equ}. Therefore, if $0<|b|<\varepsilon$, then the function
\begin{equation}\label{uubb}
U_b:=t_ve^{-i\theta}v=|t_vv|
\end{equation}
is a nonnegative real-valued ground state of \eqref{22}. By the
strong maximum principle, if $0<|b|<\varepsilon$, then we further have $U_b>0$ in $\R$.

We next prove the uniqueness, up to translations in
the $x_3$-direction, of positive real-valued ground states for \eqref{22}, where $0<|b|<\varepsilon$ is sufficiently small. Let $U_1$ and $U_2$ be two positive
real-valued ground states of \eqref{22} for sufficiently small $0<|b|<\varepsilon$. By Proposition \ref{P},
they are also ground states of \eqref{equ}. It then follows from
Theorem \ref{T1.1} that there exist
$a=(a_1,a_2,a_3)\in\R$ and $\theta\in\mathbb R$ such that
\begin{equation*}
U_1\equiv e^{i\theta}\tau_{a,A}U_2=e^{i\theta}e^{-iA(a)\cdot x}U_2(x-a) \ \ \mbox{in}\ \ \R.
\end{equation*}
Since $U_1$ and $U_2$ are positive and real-valued in $\R$, we have
\begin{equation*}
A(a)=0\text{ \ and \ }e^{i\theta}=1.
\end{equation*}
It thus follows from \eqref{1.1*} that if $0<|b|<\varepsilon$ is sufficiently small, then $a=(0,0,k)$ holds for some
$k\in\mathbb R$, and
\begin{equation*}
	U_1(x_1,x_2,x_3)\equiv U_2(x_1,x_2,x_3-k) \ \text{ in } \ \R,
\end{equation*}
which implies that positive real-valued ground states of \eqref{22} must be unique, up to translations in the $x_3$-direction.

We now claim that if $0<|b|<\varepsilon$ is sufficiently small, then every minimizer $v\in\H_b\setminus\{0\}$ of $\widetilde q(b)$ is of the form
\begin{equation}\label{qbm}
v(x_1,x_2,x_3)=Ce^{i\theta}U_b(x_1,x_2,x_3-k) \ \text{ \ for some } C>0,\theta\in\mathbb R,k\in\mathbb R,
\end{equation}
where $U_b>0$ is the unique positive ground state of \eqref{22}.
Similar to \eqref{uubb}, one can get that for every minimizer $v\in\H_b\setminus\{0\}$ of $\widetilde q(b)$, there exist
$t_v>0$ and $\theta\in\mathbb R$ such that $t_ve^{-i\theta}v$ is a positive real-valued ground state of \eqref{22}. By the
above uniqueness, if $0<|b|<\varepsilon$ is sufficiently small, then there exists $k\in\mathbb R$ such that
\begin{equation*}
	t_ve^{-i\theta}v(x_1,x_2,x_3)\equiv U_b(x_1,x_2,x_3-k) \ \text{ in } \ \R.
\end{equation*}
Consequently, if $0<|b|<\varepsilon$ is sufficiently small, then the claim (\ref{qbm}) holds for $C=t_v^{-1}>0$.

It remains to prove the symmetry and monotonicity of $U_b>0$ in $\R$. Since
$U_b>0$ is a ground state of \eqref{22}, it is also a minimizer of
$\widetilde q(b)$. Define
\begin{equation*}
	\widetilde U_b(x_1,x_2,x_3)=\big(U_b(\cdot,\cdot,x_3)\big)^*(x_1,x_2),
\end{equation*}
where $(\cdot)^*$ denotes the Schwarz rearrangement of $(\cdot)$ with respect to
$(x_1,x_2)$. Then $\widetilde{U}_b$ is decreasing in $(x_1,x_2)$. Note from \cite{Brock} that
\begin{equation}\label{sv}
\begin{aligned}
	&\int_{\R}|\nabla\widetilde U_b|^2dx
	\leq\int_{\R}|\nabla U_b|^2dx,\ \
	\int_{\R}|\widetilde U_b|^2dx
	=\int_{\R}|U_b|^2dx,\\
	&\iint_{\R\times\R}
	\f{|\widetilde U_b(x)|^2|\widetilde U_b(y)|^2}{|x-y|}dxdy
	\geq
	\iint_{\R\times\R}
	\f{|U_b(x)|^2|U_b(y)|^2}{|x-y|}dxdy.
\end{aligned}
\end{equation}
Moreover, it yields from \cite[Appendix A]{Bella} that for any $x_3\in\mathbb R$,
\begin{equation}\label{4.1}
\begin{aligned}
\int_{\mathbb R^2}(x_1^2+x_2^2)|\widetilde U_b(x_1,x_2,x_3)|^2dx_1dx_2\leq\int_{\mathbb R^2}(x_1^2+x_2^2)|U_b(x_1,x_2,x_3)|^2dx_1dx_2.
\end{aligned}
\end{equation}
Integrating \eqref{4.1} with respect to
$x_3$, we then obtain that
\begin{equation}\label{4.8}
	\int_{\R}(x_1^2+x_2^2)|\widetilde U_b|^2dx\leq\int_{\R}(x_1^2+x_2^2)|U_b|^2dx.
\end{equation}
It thus follows from \eqref{sv} and \eqref{4.8} that
\begin{equation*}
	\widetilde Q_b(\widetilde U_b)\leq\widetilde Q_b(U_b),
\end{equation*}
which implies that $\widetilde U_b$ is also a minimizer of $\widetilde q(b)$.
Hence, all identities of \eqref{sv}  and \eqref{4.8}  hold. In particular, we have
\begin{equation*}
	\int_{\R}(x_1^2+x_2^2)|\widetilde U_b|^2dx=\int_{\R}(x_1^2+x_2^2)|U_b|^2dx.
\end{equation*}
We then obtain from \eqref{4.1} that for   a.e. $ x_3\in\mathbb R$,
\begin{equation*}
\begin{aligned}
\int_{\mathbb R^2}(x_1^2+x_2^2)|\widetilde U_b(x_1,x_2,x_3)|^2dx_1dx_2=\int_{\mathbb R^2}(x_1^2+x_2^2)|U_b(x_1,x_2,x_3)|^2dx_1dx_2.
\end{aligned}
\end{equation*}
It thus yields from \cite[Theorem 4]{Bella} that
\begin{equation*}
\widetilde U_b=U_b \ \text{ a.e. in} \ \R.
\end{equation*}
Since $U_b$ is continuous in $\R$, $U_b$ is radially symmetric and decreasing in $(x_1,x_2)$ for any $x_3\in\mathbb R$.

We finally prove the symmetry and monotonicity of $U_b>0$ in $x_3$. Let $U_b^\dagger$ be the Steiner symmetrization of $U_b$ with respect to $x_3$, $i.e.$,
\begin{equation*}
	U_b^\dagger(x_1,x_2,x_3)=\big(U_b(x_1,x_2,\cdot)\big)^*(x_3),
\end{equation*}
where $(\cdot)^*$ denotes the one-dimensional Schwarz rearrangement of $(\cdot)$
with respect to $x_3$. Then $U_b^\dagger>0$ is even and decreasing with respect to
$x_3$. We then obtain from \cite{Brock} that
\begin{equation}\label{4.10}
\begin{aligned}
	&\int_{\R}|\nabla U_b^\dagger|^2dx
	\leq\int_{\R}|\nabla U_b|^2dx,\ \
	\int_{\R}|U_b^\dagger|^2dx
	=\int_{\R}|U_b|^2dx,\\
	&\iint_{\R\times\R}
	\f{|U_b(x)|^2|U_b(y)|^2}{|x-y|}dxdy
	\leq
	\iint_{\R\times\R}
	\f{|U_b^\dagger(x)|^2|U_b^\dagger(y)|^2}{|x-y|}dxdy.
\end{aligned}
\end{equation}
Since $|A(x)|^2=\f{b^2}{4}(x_1^2+x_2^2)$ is independent of $x_3$, we also have
\begin{equation*}
	\int_{\R}|A|^2|U_b^\dagger|^2dx=\int_{\R}|A|^2|U_b|^2dx.
\end{equation*}
Hence,
\begin{equation*}
	\widetilde Q_b(U_b^\dagger)\leq\widetilde Q_b(U_b),
\end{equation*}
and $U_b^\dagger>0$ is also a minimizer of $\widetilde q(b)$. It then follows from \eqref{qbm} that there exist $C>0$, $\theta\in\mathbb R$, and $k\in\mathbb R$ such that for any $0<|b|<\varepsilon$,
\begin{equation*}
	U_b(x_1,x_2,x_3)\equiv Ce^{i\theta}U_b^\dagger(x_1,x_2,x_3-k) \ \,\text{ in } \, \R.
\end{equation*}
Since $U_b$ and $U_b^\dagger$ are both positive and real-valued in $\R$, we have
$e^{i\theta}=1$. Moreover, since $\|U_b^\dagger\|_{2}=\|U_b\|_{2}$,
we have $C=1$. Thus, up to a suitable translation in
$x_3$, $U_b$ is symmetric and decreasing with respect to
$x_3$. This completes the proof of Lemma \ref{P5.1}. \qed

\vspace{5pt}

\noindent{\bf Proof of Theorem \ref{T1.2}.}
Let $U_b>0$ be the positive real-valued ground state obtained in
Lemma \ref{P5.1}, where $0<|b|<\varepsilon$ holds for some sufficiently small $\varepsilon>0$. After a suitable translation in the
$x_3$-direction, $U_b$ is radially symmetric and decreasing in $(x_1,x_2)$ and in $x_3$. By Proposition \ref{P}, $U_b$ is also a ground state of \eqref{equ}.

Let $u$ be an arbitrary ground state of \eqref{equ}. By Theorem
\ref{T1.1}, if $0<|b|<\varepsilon$ holds for sufficiently small $\varepsilon>0$, then there exist $\theta\in\mathbb R$ and $a\in\mathbb R^3$
such that
\begin{equation*}
u\equiv e^{i\theta}\tau_{a,A}U_b \ \, \text{ in } \, \R.
\end{equation*}
This therefore completes the proof of Theorem \ref{T1.2}. \qed

\section{Refined Asymptotic Expansion of   $e(b)$ as $b\to0$}\label{S5}

This section is devoted to the proof of Theorem \ref{T1.4}, which concerns the refined asymptotic expansion of
the ground state energy $e(b)$ defined by \eqref{inf} as $b\to0$.

Let $u_A>0$ be the real-valued even ground state obtained in Theorem \ref{T1.2}, where $A(x)=\f b2(-x_2,x_1,0)$. Suppose $u_0(x)=u_0(|x|)>0$ is the unique positive ground state of the limiting problem \eqref{w}. The proof of Theorem \ref{T1.4} needs the following crucial proposition.

\begin{Proposition}\label{7.1}
Suppose $A(x)=\f b2(-x_2,x_1,0)$. Then we have
\begin{equation*}
\|u_A-u_0-w_A\|_{H^1}=o(b^2) \text{ \ as \ }b\to0,
\end{equation*}
where $w_A\in H^1(\R,\mathbb C)$ is the unique solution of
\begin{equation}\label{wA}
\left\{
\begin{aligned}
&(-\Delta+1)w_A-w_A\big(|x|^{-1}*|u_0|^2\big)-2u_0\big(|x|^{-1}*(u_0|w_A)\big)=-|A|^2u_0\text{  \ in \ }\R,\\
&\bigl(\nabla u_0\cdot y|w_A\bigr)_{H^1}
=0\text{ \ for every \ }y\in\mathbb R^3,\\
&\bigl(iu_0|w_A\bigr)_{H^1}=0.
\end{aligned}
\right.
\end{equation}
Moreover, $w_A$ is real-valued and even in $\R$.
\end{Proposition}

In order to prove Proposition \ref{7.1}, we shall establish the following several lemmas.

\begin{Lemma}\label{L7.1}
Suppose $f\in L^2(\R,\mathbb C)$ satisfies
\begin{equation}\label{ff}
\left\{
\begin{aligned}
&\int_{\R}(\nabla u_0\cdot y|f)dx=0\text{ \ for every \ }y\in\mathbb R^3,\\
&\int_{\R}(iu_0|f)dx=0,
\end{aligned}
\right.
\end{equation}
where $u_0=u_0(|x|)>0$ is the unique positive ground state of \eqref{w}.
Then the problem
\begin{equation}\label{v}
\left\{
\begin{aligned}
&-\Delta v+v-v\big(|x|^{-1}*|u_0|^2\big)-2u_0\big(|x|^{-1}*(u_0|v)\big)
=f \text{ \ in \ }\R,\\
&\bigl(\nabla u_0\cdot y|v\bigr)_{H^1}=0\text{ \ for every \ }y\in\mathbb R^3,\\
&\bigl(iu_0|v\bigr)_{H^1}=0
\end{aligned}
\right.
\end{equation}
has a unique solution $v\in H^1(\R,\mathbb C)$, which satisfies
\begin{equation}\label{7.35}
\|v\|_{H^1}\leq C\|f\|_{2}.
\end{equation}
\end{Lemma}

\noindent{\bf Proof.}
Define the operator $L: H^1(\R,\mathbb{C})\to H^1(\R,\mathbb{C}) $ by
\begin{equation*}
	Lu=(-\Delta+I)^{-1}\Big[u\big(|x|^{-1}*|u_{0}|^2\big)+2u_{0}\big(|x|^{-1}*(u_{0}|u)\big)\Big] \ \text{ in } \ \R.
\end{equation*}
Similar to Lemma \ref{LLn}, one can verify that $L$ is self-adjoint and compact in $H^1(\R,\mathbb{C})$.
We then rewrite the equation of \eqref{v} as
\begin{equation}\label{7.352}
	(I-L)v=(-\Delta+I)^{-1}f\ \text{ \ in \ }\R.
\end{equation}
It yields from Lemma \ref{P2} that
\begin{equation*}
\ker(I-L)=\operatorname{span}_{\mathbb R} \{iu_0,\partial_{x_1}u_0,\partial_{x_2}u_0, \partial_{x_3}u_0\}.
\end{equation*}
By Fredholm's alternative theorem (cf. \cite{Evans}), the orthogonality condition of \eqref{ff} then yields that the equation \eqref{v} admits a solution $v \in (\ker(I-L))^{\bot}$ in $ H^1(\R,\mathbb{C})$.

We next prove the uniqueness of solutions for \eqref{v}. Let $v_1$ and $v_2$ be two different solutions
of \eqref{v}, and set $w:=v_1-v_2$. We then have
\begin{equation*}
(I-L)w=0\ \text{ \ in \ }\R,
\end{equation*}
and $w\in\ker(I-L)^\perp$. Hence,
\begin{equation*}
w\in\ker(I-L)\cap\ker(I-L)^\perp,
\end{equation*}
which implies that $w\equiv 0$ in $\R$. Therefore, the solutions of \eqref{v} must be
unique.

It remains to prove the estimate \eqref{7.35}. Since $L$ is self-adjoint and
compact in $H^1(\R,\mathbb{C})$, $(I-L)|_{\ker(I-L)^\perp}$ is bounded and invertible onto $\operatorname{Range}(I-L)$.
It thus follows from \eqref{7.352} that
\begin{equation*}
\begin{aligned}
\|v\|_{H^1}
&=\|(I-L)^{-1}[(I-L)v]\|_{H^1}
 \leq C\|(I-L)v\|_{H^1}\\
&=C\|(-\Delta+I)^{-1}f\|_{H^1}\\
&\leq C\sqrt{\|(-\Delta+I)^{-1}f\|_{2}\|f\|_{2}
} \leq C\|f\|_{2},
\end{aligned}
\end{equation*}
and hence \eqref{7.35} holds true. This therefore completes the proof of Lemma \ref{L7.1}. \qed

\vspace{5pt}

\begin{Lemma}\label{L7.2}
The problem \eqref{wA} has a unique solution
$w_A\in H^1(\R,\mathbb C)$, which   is real-valued and even in $\R$.
\end{Lemma}

\noindent{\bf Proof.}
Since $u_0=u_0(|x|)$ and $|A|^2$ are real-valued and even in $\R$, we have
\begin{equation*}
\left\{
\begin{aligned}
&\int_{\R}
(\nabla u_0\cdot y\big||A|^2u_0)dx=0\text{ \ for every \ }y\in\mathbb R^3,\\
&\int_{\R}(iu_0\big||A|^2u_0)dx=0.
\end{aligned}
\right.
\end{equation*}
We then derive from Lemma \ref{L7.1} that \eqref{wA} admits a solution
$w_A\in H^1(\R,\mathbb{C})$. Since both $w_A(-x)$ and $\overline{w_A(x)}$ solve
\eqref{wA}, the uniqueness of Lemma \ref{L7.1} implies that
\begin{equation*}
w_A(-x)\equiv w_A(x)
\quad\text{and}\quad
\overline{w_A(x)}\equiv w_A(x)\quad\text{in }\ \R.
\end{equation*}
Thus, $w_A$ is even and real-valued in $\R$, and we are done. \qed

\vspace{5pt}

We also need the following uniform decaying estimate of ground states for \eqref{equ} as $b\to 0$.

\begin{Lemma}\label{L7.3}
Let $u_A>0$ be the even ground state obtained in Theorem \ref{T1.2}, where $A(x)=\f b2(-x_2,x_1,0)$. Then we have
\begin{equation}\label{uniform}
u_A\to u_0 \text{ \ strongly in \ }L^\infty(\R)\text{  \ as \ }b\to0.
\end{equation}
where $u_0=u_0(|x|)>0$ is a unique ground state of \eqref{w}.
Moreover, there exists a sufficiently small $\varepsilon>0$
such that for any $0<|b|<\varepsilon$,
\begin{equation}\label{Kato-D}
|u_A(x)|\leq Ce^{-(1-\alpha)|x|} \text{ \ in \ }\R,
\end{equation}
where     $C>0$ and $0<\alpha<1$ are independent of $b$.
\end{Lemma}

\noindent{\bf Proof.}
It follows from \cite[Lemma 2.5]{DB} and Proposition \ref{P3.1} that
 \begin{equation}\label{stron}
 u_A\to u_{0}\text{ \  in \ }L^p(\R)  \text{ \ as \ } b\to 0,\ 2\leq p\leq 6.
\end{equation}
By the
standard $L^p$ estimates and Sobolev embedding theorem (cf. \cite{GT}), it then yields that
\begin{equation}\label{5.9}
u_A\to u_0
\quad\text{strongly in } L^\infty_{\loc}(\R)\text{ \ as \ }b\to0.
\end{equation}

We next prove the uniform decaying property of $u_A$ as $|x|\to \infty$. For sufficiently small $0<|b|<\varepsilon$, it follows from  \eqref{stron} that for any $\eta>0$, there exists
$R>0$ such that
\begin{equation}\label{5.95}
\int_{\R\setminus B_R}\big(|u_A|^2+|u_A|^6\big)dx\leq\eta.
\end{equation}
We then obtain that for $|x|\geq2R$,
\begin{equation*}
\begin{aligned}
|x|^{-1}*|u_A|^2\leq&C\Big(\int_{\R\setminus B_R}|u_A|^6dx\Big)^{\f13}+\int_{\R\setminus B_R}|u_A|^2dx+\f1R\|u_A\|_2^2.
\end{aligned}
\end{equation*}
Taking sufficiently small $\eta>0$  and   sufficiently large $R>0$, we then obtain that if $|b|>0$ is sufficiently small, then we have
\begin{equation*}
|x|^{-1}*|u_A|^2\leq\alpha\ \text{ \ uniformly for \ }|x|\geq2R.
\end{equation*}
By Kato's inequality, we also have
\begin{equation}\label{Kato}
-\Delta|u_A|+(1-\alpha)|u_A|\leq0\ \text{ \ for \ }|x|\geq2R.
\end{equation}
Applying De Giorgi-Nash-Moser theory (cf.\cite[Theorem 4.1]{QH}), one can deduce from \eqref{5.95} that
\begin{equation}\label{Kato-B}
|u_A(x)|\to0\text{ \ as \ }|x|\to\infty\text{\ \   for sufficiently small \ }|b|>0.
\end{equation}
The comparison principle then yields from (\ref{Kato}) and (\ref{Kato-B}) that there exists a sufficiently small $\delta>0$
such that for any $0<|b|<\delta$,
\begin{equation}\label{Kato-C}
|u_A(x)|\leq Ce^{-(1-\alpha)|x|}\ \text{ \ in \ }\R,
\end{equation}
where $C>0$ and $0<\alpha<1$ are independent of $b$. This proves (\ref{Kato-D}).
By the exponential decay
of $u_0$, we further obtain from \eqref{5.9} and (\ref{Kato-C}) that \eqref{uniform} holds true.  This therefore completes the proof of Lemma \ref{L7.3}.\qed

\vspace{5pt}

By the dominated convergence theorem, one can derive from Lemma \ref{L7.3} that  the even ground state $u_A>0$   obtained in Theorem \ref{T1.2} satisfies
\begin{equation}\label{weighted}
\big\|(x_1^2+x_2^2)(u_A-u_0)\big\|_{2}\to0\text{ \ as \ }b\to0,
\end{equation}
where $u_0=u_0(|x|)>0$ is the unique positive ground state of \eqref{w}.
We are now ready to prove Proposition \ref{7.1}.

\vspace{5pt}

\noindent{\bf Proof of Proposition \ref{7.1}.}
Set
\begin{equation}\label{5.15A}
h_A:=u_A-u_0,\ \  v_A:=u_A-u_0-w_A=h_A-w_A.
\end{equation}
Since $u_A$ is radially symmetric in $(x_1,x_2)$ for sufficiently small $|b|>0$, we have $A\cdot\nabla u_A=0$ for sufficiently small $|b|>0$. Using the equations of $u_A$ and $u_0$, it then follows from \eqref{wA} that  for sufficiently small $|b|>0$,
\begin{equation}\label{25}
\begin{aligned}
&-\Delta v_A+v_A-v_A\big(|x|^{-1}*|u_0|^2\big)-2u_0\big(|x|^{-1}*(u_0|v_A)\big)\\
=&u_0\big(|x|^{-1}*|h_A|^2\big)+2h_A\big(|x|^{-1}*(u_0|h_A)\big)\\
&\quad+h_A\big(|x|^{-1}*|h_A|^2\big)
-|A|^2h_A:=R_A\text{ \ in \ }\R .
\end{aligned}
\end{equation}
Since $u_A$, $u_0$, and $w_A$ are real-valued and even in $\R$, the function
$v_A$ and the right-hand side $R_A$ of \eqref{25} are also real-valued and
even in $\R$. It then follows that for every $y\in\mathbb R^3$,
\begin{equation*}
\bigl(\nabla u_0\cdot y|v_A\bigr)_{H^1}=0,
\ \
\bigl(iu_0|v_A\bigr)_{H^1}=0
\end{equation*}
 and hence the right-hand side $R_A$ of \eqref{25}
satisfies \eqref{ff}.

By H\"older's inequality, we obtain from \eqref{uniform} that
\begin{equation*}
\big\||x|^{-1}*|h_A|^2\big\|_{\infty}=o(1)\|u_A-u_0\|_{2},\ \
\big\||x|^{-1}*(u_0|h_A)\big\|_{\infty}=o(1)\text{ \ as \ }b\to0,
\end{equation*}
where $h_A$ is as in \eqref{5.15A}.
The right-hand side $R_A$ of \eqref{25} then satisfies
\begin{equation}\label{6.20M}
\begin{aligned}
\|R_A\|_{2}\leq& o(1)\|h_A\|_{2}+\big\||A|^2h_A\big\|_{2}
\leq o(1)\|h_A\|_{H^1}+o(b^2)\text{ \ as \ }b\to0,
\end{aligned}
\end{equation}
where we have used \eqref{weighted} in the second inequality.

Applying Lemma \ref{L7.1}, we deduce from \eqref{wA} that
\begin{equation}\label{6.20}
\|w_A\|_{H^1}\leq C\big\||A|^2u_0\big\|_{2}
\leq Cb^2\text{ \ as \ }b\to0,
\end{equation}
and we obtain from  \eqref{25} and (\ref{6.20M}) that for $h_A=w_A+v_A$,
\begin{equation*}
\begin{aligned}
\|v_A\|_{H^1}
&\leq C\|R_A\|_{2}
\leq o(1)\big(\|w_A\|_{H^1}+\|v_A\|_{H^1}\big) +o(b^2)\text{ \ as \ }b\to0.
\end{aligned}
\end{equation*}
We then derive from \eqref{6.20} that
\begin{equation}\label{6.21}
\|v_A\|_{H^1}=o(b^2)\text{ \ as \ }b\to0.
\end{equation}
The proof of Proposition \ref{7.1} is therefore complete. \qed

\vspace{5pt}

Applying Proposition \ref{7.1}, we are now ready to prove Theorem
\ref{T1.4}.

\vspace{5pt}

\noindent{\bf Proof of Theorem \ref{T1.4}.}
Recall from \eqref{6.20} and \eqref{6.21} that
\begin{equation}\label{6.28}
\|w_A\|_{H^1}=O(b^2),
\quad
\|v_A\|_{H^1}=o(b^2)\text{ \ as \ }b\to0,
\end{equation}
where $u_A=u_0+w_A+v_A$.
Since $u_A$ is real-valued in $\R$, we have
\begin{equation*}
|\nabla_Au_A|^2=|\nabla u_A|^2+|A|^2|u_A|^2 \ \ \ \mbox{in} \ \ \R.
\end{equation*}
Therefore,
\begin{equation}\label{5.15}
e(b)=I_0(u_A)+\f12\int_{\R}|A|^2|u_A|^2dx.
\end{equation}
Since $I_0'(u_0)=0$, it follows from \eqref{6.28} that
\begin{equation}\label{5.16}
\begin{aligned}
I_0(u_A)
&=I_0(u_0)
+\big\langle I_0'(u_0),w_A+v_A\big\rangle
+O\bigl(\|w_A+v_A\|_{H^1}^2\bigr)\\
&=e(0)+O(b^4)
=e(0)+o(b^2)\text{ \ as \  }b\to0.
\end{aligned}
\end{equation}
Moreover, we obtain from \eqref{weighted} that
\begin{equation}\label{5.17}
\int_{\R}|A|^2|u_A|^2dx=\int_{\R}|A|^2|u_0|^2dx+o(b^2) \ \text{ as }\, b\to0.
\end{equation}
We thus obtain from \eqref{5.15}--\eqref{5.17} that
\begin{equation}\label{27}
e(b)=e(0)+\f12\int_{\R}|A|^2|u_0|^2dx+o(b^2)\text{ \ as }\, b\to0.
\end{equation}
Since $u_0$ is radially symmetric in $\R$, we have
\begin{equation*}
\begin{aligned}
\int_{\R}(x_1^2+x_2^2)|u_0|^2dx&=\f23\int_{\R}|x|^2|u_0|^2dx.
\end{aligned}
\end{equation*}
Consequently,
\begin{equation*}
\begin{aligned}
\int_{\R}|A|^2|u_0|^2dx=\f{b^2}{4}\int_{\R}(x_1^2+x_2^2)|u_0|^2dx=\f{b^2}{6}
\int_{\R}|x|^2|u_0|^2dx.
\end{aligned}
\end{equation*}
Substituting this identity into \eqref{27}, we then obtain \eqref{ee}.
This completes the proof of Theorem \ref{T1.4}. \qed

\end{document}